\documentclass[11pt,reqno]{amsart}
\usepackage{amsmath, amsfonts, amssymb, amsthm, amsxtra, array, stmaryrd, graphicx, hyperref}
\usepackage{textcomp, verbatim, color, mathrsfs, mathabx, braket}
\usepackage{epsfig, graphicx, caption}
\usepackage[margin=1in]{geometry}
\newcommand{\la}{\langle}
\newcommand{\ra}{\rangle}

\newcommand{\pphi}{|\phi\ra\la\phi|}

\newtheorem{thm}{Theorem}[section]
\newtheorem{lemma}{Lemma}[section]

\newtheorem{prop}[lemma]{Proposition}

\newtheorem*{acklg}{Acknowledgements}
\numberwithin{equation}{section}
\allowdisplaybreaks

\title{Unconditional Uniqueness of the cubic Gross-Pitaevskii Hierarchy with Low Regularity}
\author{ Younghun Hong }
\address{Department of Mathematics \newline\indent The Uxniversity of Texas at Austin}
\email{yhong@math.utexas.edu}

\author{ Kenneth Taliaferro }
\address{Department of Mathematics \newline\indent The University of Texas at Austin}
\email{ktaliaferro@math.utexas.edu}

\author{ Zhihui Xie }
\address{Department of Mathematics \newline\indent The University of Texas at Austin}
\email{zxie@math.utexas.edu}

\date{\today}

\begin{document}
  
\begin{abstract}

In this paper, we establish the unconditional uniqueness of solutions to the cubic Gross-Pitaevskii hierarchy on $\mathbb{R}^d$ in a low regularity Sobolev type space. More precisely, we reduce the regularity $s$ down to the currently known regularity requirement for unconditional uniqueness of solutions to the cubic nonlinear Schr\"odinger equation ($s\ge\frac{d}{6}$ if $d=1,2$ and $s>s_c=\frac{d-2}{2}$ if $d\ge 3$). In such a way, we extend the recent work of Chen-Hainzl-Pavlovi\'c-Seiringer \cite{CHPS}.

\end{abstract}

\maketitle

\section{Introduction}   \label{sec: introduction}

\subsection{Background}

The cubic Gross-Pitaevskii (GP) hierarchy in $\mathbb{R}^d$ is an infinite system of coupled linear equations given by
\begin{equation}\label{GP}
    i\partial_t \gamma^{(k)}=(-\Delta_{\underline{x}_k}+\Delta_{\underline{x}_k'})\gamma^{(k)}+ 
    \lambda B_{k+1}\gamma^{(k+1)},\quad\forall k\in\mathbb{N},
\end{equation}
where $\gamma^{(k)}=\gamma^{(k)}(t, \underline{x}_k,\underline{x}_k'): I\times\mathbb{R}^{dk}\times\mathbb{R}^{dk}\to\mathbb{C}$, $I\subset \mathbb{R}$ is a time interval and $\lambda=\pm1$. Here, we denote $d$-dimensional $k$-spatial variables $(x_1, x_2,..., x_k)$ by $\underline{x}_k$, and the corresponding Laplace operator by $\Delta_{\underline{x}_k}=\sum_{j=1}^k\Delta_{x_j}$, and similarly for the primed variables. For each $k\in\mathbb{N}$, $\gamma^{(k)}$ is a bosonic density matrix on $L_{sym}^2(\mathbb{R}^{dk})$ which is hermitian,
$$
    \gamma^{(k)}(t,\underline{x}_k,\underline{x}_k')=\overline{\gamma^{(k)}
    (t,\underline{x}_k',\underline{x}_k)},
$$ 
and is symmetric in all components of $\underline{x}_k$, and in all components of $\underline{x}_k'$, respectively, $$\gamma^{(k)}(t,x_{\sigma(1)},\cdots,x_{\sigma(k)},x_{\sigma'(1)}',\cdots,x_{\sigma'(k)}')=\gamma^{(k)}(t,\underline{x}_k, \underline{x}_k')$$
for any permutations $\sigma, \sigma'$ on $k$ elements. The equations in \eqref{GP} are coupled by the \emph{contraction operator} $B_{k+1}$,
$$B_{k+1}=\sum_{j=1}^k B_{j;k+1}=\sum_{j=1}^k(B^{+}_{j;k+1}-B^{-}_{j;k+1}),$$
where each $B^{+}_{j;k+1}$ contracts the triple $x_j,x_{k+1},x'_{k+1}$,
\begin{align*}
   \Big(B^{+}_{j;k+1}\gamma^{(k+1)} \Big)(t,\underline{x}_k,\underline{x}'_k)&=
   \int dx_{k+1}dx'_{k+1}\delta(x_j-x_{k+1}) \delta (x_j-x'_{k+1}) \gamma^{(k+1)} (t,\underline{x}_{k+1};    
   \underline{x}'_{k+1})\\ 
    &=\gamma^{(k+1)}(t, \underline{x}_k, x_j, \underline{x}_k', x_j)
\end{align*}
and each $B^{-}_{j;k+1}$ contracts the triple $x'_j,x_{k+1},x'_{k+1}$,
\begin{align*}
    \Big(B^{-}_{j;k+1}\gamma^{(k+1)} \Big)(t,\underline{x}_k,\underline{x}'_k)&= 
    \int dx_{k+1}dx'_{k+1}\delta(x'_j-x_{k+1}) \delta (x'_j-x'_{k+1}) \gamma^{(k+1)} (t,\underline{x}_{k+1};    
    \underline{x}'_{k+1})\\
    &=\gamma^{(k+1)}(t, \underline{x}_k, x_j', \underline{x}_k', x_j').
\end{align*}
The cubic GP hierarchy is called \textit{focusing} (\textit{defocusing}, respectively) if $\lambda=1$ ($\lambda=-1$, respectively).\\

The cubic GP hierarchy is an infinite hierarchy of equations modeling a Bose-Einstein condensate.
For the mathematical study of Bose-Einstein condensation (BEC)
in systems of interacting bosons in the stationary case, we refer to the fundamental works \cite{LS,LSY,LSY2000,LSSY} and the references therein. 
To study the dynamics of Bose-Einstein condensates, one considers
$N$ bosonic particles whose quantum mechanical wave function 
$\psi_{N}\in L^2_{sym}({\mathbb R}^{dN})$
satisfies the $N$-body Schr\"odinger equation
\begin{equation}   \label{LS}
    i\partial_t \psi_{N}=H_N \psi_{N},
\end{equation}
where
$$H_N= \sum_{j=1}^N (-\Delta_{x_j}) +\frac{1}{N}\sum_{1\leq i<j\leq N} V_N(x_i-x_j)$$
and $V_N(x)=N^{d\beta}V(N^\beta x)$ with $\beta\in(0,1)$ (we remark that the case $\beta=1$ is much more difficult to control  \cite{ESY06, ESY07, ESY09, ESY10}). The pair interaction potential $V$ is assumed to be  rotationally symmetric, and to satisfy certain regularity properties. The cubic GP hierarchy is then formally obtained from a limit of the BBGKY hierarchy of marginal density matrices associated to \eqref{LS} as $N\to\infty$. In this limit, $V_N$ converges weakly to  $(\int V(x) dx)\delta$, where $\delta$ denotes the delta distribution. In this sense, the cubic GP hierarchy describes a Bose gas of infinitely many particles with repulsive or attractive two-body delta interactions.

In the special case of factorized initial data $\gamma^{(k)}_0(\underline{x}_k,\underline{x}'_k)=\prod_{j=1}^k \phi_0(x_j)\overline{\phi_0}(x'_j)$ in \eqref{GP}, the state of a Bose-Einstein condensate can be simply described by the cubic nonlinear Schr\"odinger equation (NLS). Indeed, in this case, the cubic GP hierarchy admits a solution
$$\gamma^{(k)}(t,\underline{x}_k,\underline{x}'_k)=\prod_{j=1}^k \phi(t,x_j)\overline{\phi}(t,x'_j),$$
preserving the factorization property as time evolves, if $\phi$ solves the cubic NLS
\begin{equation}   \label{NLS}
    i\partial_t\phi=-\Delta\phi+\lambda|\phi|^2\phi,\ \phi(0)=\phi_0.
\end{equation}
In this way, the cubic NLS is derived as a dynamical mean field limit of the many body quantum dynamics of an interacting Bose gas, provided that given initial data, a solution to the GP hierarchy is unique. We call this formal derivation the BBGKY approach.
In his fundamental works \cite{Lanford1,Lanford2}, Lanford had employed the BBGKY hierarchy to study $N$-body systems in classical mechanics in the limit $N\rightarrow\infty$.  
 
Research efforts aimed at providing a rigorous derivation of nonlinear dispersive equations as mean field limits of  $N$-body Schr\"odinger dynamics have a long and rich history. 
The first results on the derivation of nonlinear Hartree equations (NLH) were due to Hepp \cite{Hepp}, and Ginibre and Velo \cite{GiniVelo1,GiniVelo2}. Their techniques are based on embedding the $N$-body Schr\"odinger equation into the second quantized Fock-space representation. 
In \cite{Spohn} Spohn gives the first derivation of NLH by use of the BBGKY hierarchy. 
More recently, Erd\"os, Schlein and Yau further developed the BBGKY approach, and gave the first derivation of NLS in their celebrated works \cite{ESY06, ESY07, ESY09, ESY10}. In \cite{SchleinRodnianski}, Rodnianski and Schlein proved estimates on the convergence rate of the evolution in the mean field 
limit using the Fock space approach. Their results were extended with second-order corrections in the two-body interaction setting by Grillakis, Machedon and Margetis \cite{GMM2, GMM1}, and three-body interaction setting by X. Chen \cite{Xuwen}.

The  derivation of the cubic NLS in ${\mathbb R}^3$, via the BBGKY approach, due to Erd\"os, Schlein and Yau \cite{ESY06, ESY07, ESY09, ESY10}, comprises the following two main parts:
\begin{itemize}
\item
[(i)] Derivation of the GP hierarchy as the limit of the $N$-body BBGKY hierarchy as $N\to \infty$.
\item
[(ii)] Establishing the uniqueness of solutions to the GP hierarchy. In particular, it is proved that for factorized initial data, the solutions to the GP hierarchy are determined by a cubic NLS.  
\end{itemize} 
In this program, the proof of the uniqueness theorem (part (ii)) is very involved, one of the difficulties being the factorial growth of the number of terms from iterated Duhamel expansions. The authors give a sophisticated combinatorial argument that settled this problem by a clever re-grouping of Feynman graph expansions.\\

Later, in \cite{KM},  Klainerman and Machedon gave a shorter proof of uniqueness of solutions to the 3D cubic GP hierarchy in a different solution space, provided that solutions obey a priori bound,
\begin{equation}   \label{KM bound}
 \int_0^T \| R^{(k)}B_{j;k+1} \gamma^{(k+1)}(t,\cdot,\cdot)\|_{L^2(\mathbb{R}^{dk}\times \mathbb{R}^{dk})} dt < C^k,\quad\forall k\in\mathbb{N},
\end{equation}
where $R_j=(-\Delta_{x_j})^{1/2}$, $R_j'=(-\Delta_{x_j'})^{1/2}$ and $R^{(k)}=\prod_{j=1}^k R_j\prod_{j=1}^k R_j'$. The approach is in part motivated by the authors' previous work on the space-time estimates \cite{KM93}. In \cite{KM}, Klainerman and Machedon gave a concise reformulation of the Erd\"os-Schlein-Yau combinatorial method \cite{ESY06, ESY07, ESY09, ESY10}, and presented it as an elegant board game argument. The uniqueness theorem of \cite{KM} is \textit{conditional} due to the hypothesis \eqref{KM bound}. Since the work \cite{KSS} for the cubic GP hierarchy on two dimensional Euclidean space as well as the 2-dimensional torus, the approach of Klainerman and Machedon was used in various recent works for the derivation of the NLS from interacting Bose gases \cite{CPquintic, CPcubic13, xuwen3dcubic, Chen-Holmer, KSS, ChenKenny, UniqueGP}. The method also inspired the analysis of the Cauchy problem for the GP hierarchy, which was initiated in \cite{CPcauchy} and continued e.g. in \cite{GSS,ChenKenny}.
 \\

We will call the uniqueness of solutions to the GP hierarchy \textit{unconditional} if it holds without assuming any a priori bound of the form \eqref{KM bound}. 
Recently, in \cite{CHPS}, Chen-Hainzl-Pavlovi\'c-Seiringer presented a new, simpler proof of the unconditional uniqueness of solutions to the 3D cubic GP hierarchy, which is equivalent to the uniqueness result of Erd\"os-Schlein-Yau \cite{ESY07}.  The authors employed the quantum de Finetti theorem (Theorem \ref{thm: strong de Finetti} and \ref{de Finetti}) combined with the Erd\"os-Schlein-Yau combinatorial method \cite{ESY06, ESY07, ESY09, ESY10} in board game representation as presented by Klainerman-Machedon in \cite{KM}.

\subsection{Main result} \label{sec: main results}

In this paper, we investigate the unconditional uniqueness of solutions to the cubic GP hierarchy in a low regularity setting.\\ 

To state the main theorem, we first introduce the following definitions. Let $\{\gamma^{(k)}\}_{k\in\mathbb{N}}$ be a sequence of bosonic density matrices on $L_{sym}^2(\mathbb{R}^{dk})$. We say that $\{\gamma^{(k)}\}_{k\in\mathbb{N}}$ is \textit{admissible} if $\mathcal{\gamma}^{(k)}$ is a non-negative trace class operator on $L_{sym}^2(\mathbb{R}^{dk})$ and $\gamma^{(k)}=\textup{Tr}(\gamma^{(k+1)})$ for all $k\in\mathbb{N}$. We call a sequence $\{\gamma^{(k)}\}_{k\in\mathbb{N}}$ a \textit{limiting hierarchy} if there is a sequence $\{\gamma_N^{(N)}\}_{N\in\mathbb{N}}$ of non-negative density matrices on $L_{sym}^2(\mathbb{R}^{dN})$ with $\textup{Tr}(\gamma_N^{(N)})=1$ such that $\gamma^{(k)}$ is the weak-* limit of the $k$-particle marginals of $\gamma_N^{(N)}$ in the trace class on $L_{sym}^2(\mathbb{R}^{dk})$, that is,
$$\gamma_N^{(k)}:=\textup{Tr}_{k+1,\cdot\cdot\cdot, N}(\gamma_N^{(N)})\rightharpoonup^*\gamma^{(k)}\textup{ as }N\to\infty.$$
For $s\in\mathbb{R}$, we define the function space $\mathfrak{H}^s$ by the collection of sequences $\{\gamma^{(k)}\}_{k\in\mathbb{N}}$ of density matrices on $L_{sym}^2(\mathbb{R}^{dk})$ such that
$$\text{Tr}(|S^{(k,s)}\gamma^{(k)}|) < M^{2k}\quad \forall k\in\mathbb{N}\textup{ for some constant }M>0,$$
where
$$S^{(k,s)}:=\prod_{j=1}^k (1-\Delta_{x_j})^{\frac{s}{2}}(1-\Delta_{x'_j})^{\frac{s}{2}}.$$
We say that $\{\gamma^{(k)}(t)\}_{k\in\mathbb{N}}$ is a \textit{mild solution}, in the space $L_{t\in[0,T)}^\infty \mathfrak{H}^s$, to the cubic GP hierarchy with initial data $\{\gamma^{(k)}(0)\}_{k\in\mathbb{N}}$ if it solves the integral equation
$$ \gamma^{(k)}(t)=U^{(k)}(t)\gamma^{(k)}(0)+i\lambda\int_0^t U^{(k)}(t-s)B_{k+1}\gamma^{(k+1)}(s)ds,$$
where $U^{(k)}(t):=e^{it(\Delta_{\underline{x}_k}-\Delta_{\underline{x}'_k})}$, and satisfies the bound 
$$\sup_{t\in[0,T)}\text{Tr}(|S^{(k,s)}\gamma^{(k)}(t)|) < M^{2k}\quad \forall k\in\mathbb{N}\textup{ for some constant }M>0.$$

Our main theorem states that any mild solution to the cubic GP hierarchy, which is either admissible or a limiting hierarchy, is unconditionally unique in $L_{t\in[0,T)}^\infty \mathfrak{H}^s$ for small $s$.
\begin{thm}[Unconditional uniqueness]\label{thm: main theorem}
Let 
\begin{equation}\label{eq:s}
\left\{\begin{aligned}
s\geq&\tfrac{d}{6}&&\textit{ if } d=1,2,\\
s>&s_c&&\textit{ if } d\geq3,
\end{aligned}
\right.
\end{equation}
where $s_c=\tfrac{d-2}{2}$. If $\{\gamma^{(k)}(t)\}_{k\in\mathbb{N}}$ is a mild solution in $L_{t\in[0,T)}^\infty\mathfrak{H}^s$ to the (de)focusing cubic GP hierarchy with initial data $\{\gamma^{(k)}(0)\}_{k\in\mathbb{N}}$, which is either admissible or a limiting hierarchy for each $t$, then it is the only such solution for the given initial data.
\end{thm}

Our theorem reduces the regularity requirement for unconditional uniqueness for the GP hierarchy in \cite{CHPS}.  We remark that the regularity assumption in \eqref{eq:s} is the same as in the currently known unconditional uniqueness results for the cubic NLS
$$i \partial_t \phi+\Delta \phi-\lambda|\phi|^2\phi=0,\ \phi(0)=\phi_0\in H^s.$$
For NLS, by unconditional uniqueness, we mean uniqueness of solutions in the Sobolev space $H^s$ itself, while uniqueness in the intersection of the Sobolev space and auxiliary spaces is called conditional. By the contraction mapping argument with auxiliary Strichartz spaces, the conditional uniqueness is proved in $H^s$ for $s\geq \max(s_c,0)$, where $s_c=\tfrac{d-2}{2}$ (see \cite{Cazenave}). However, the unconditional uniqueness is proved in $H^s$ only for $s$ in \eqref{eq:s}, and it is an open problem to push $s$ down to zero in one and two dimensions \cite{Kato95,FTBesov03,Roger07,WinTsutsumi08,HanFang}.

Our proof uses the Klainerman-Machedon board game formulation \cite{KM} of the combinatorial argument of Erd\"os-Schlein-Yau \cite{ESY06, ESY07, ESY09, ESY10}, and the method of Chen-Hainzl-Pavlovi\'c-Seiringer \cite{CHPS} via the quantum de Finetti theorem.

The quantum de Finetti theorem is a quantum analogue of the Hewitt-Savage theorem in probability theory. We state its strong and weak versions in the formulation of \cite{LNR}.

\begin{thm}[Strong quantum de Finetti theorem] \label{thm: strong de Finetti}
If a sequence $\{\gamma^{(k)}\}_{k\in\mathbb{N}}$ of bosonic density matrices on $L_{sym}^2(\mathbb{R}^{dk})$ is admissible, then there exists a unique Borel probability measure $\mu$, supported on the unit sphere $S\subset L^2(\mathbb{R}^d)$ and invariant under multiplication of $\phi\in L^2(\mathbb{R}^d)$ by complex numbers of modulus one, such that
\begin{equation}  \label{de finetti cond}
\gamma^{(k)}=\int d\mu(\phi) (|\phi\ra\la\phi|)^{\otimes k}\quad k\in \mathbb{N}.
\end{equation}
\end{thm}

\begin{thm}[Weak quantum de Finetti theorem]\label{de Finetti}
If a sequence $\{\gamma^{(k)}\}_{k\in\mathbb{N}}$ of bosonic density matrices on $L_{sym}^2(\mathbb{R}^{dk})$ is a limiting hierarchy, then there exists a unique Borel probability measure $\mu$, supported on the unit ball $\mathcal{B}\subset L^2(\mathbb{R}^d)$ and invariant under multiplication of $\phi\in L^2(\mathbb{R}^d)$ by complex numbers of modulus one, such that \eqref{de finetti cond} holds.
\end{thm}

The crucial advantage of using the quantum de Finetti theorem is that it provides a factorized representation of solutions to the GP hierarchy in the integral form (see $(\ref{de finetti form})$). This structure allows us to make use of techniques of NLS theory to analyze solutions to the GP hierarchies (see \cite{CHPS} and \cite{CHPSscat}).\\

As described in Section 6.1.1 of \cite{CHPS}, the main difficulty in lowering regularity is from the last cubic term $\||\phi|^2\phi\|_{L^2}=\|\phi\|_{L^6}^3$ in the distinguished tree. Indeed, this last term can be controlled by the Sobolev norm $\|\phi\|_{H^s}^3$ only for $s\geq 1$ in $\mathbb{R}^3$. We solve this problem by using the dispersive estimate
$$\|e^{it\Delta}f\|_{L^{\frac{6}{1+2\epsilon}}}\lesssim |t|^{-(1-\epsilon)}\|f\|_{L^{\frac{6}{5-2\epsilon}}}$$
in $\mathbb{R}^3$, for instance. Indeed, if one applies the dispersive estimate and the endpoint Strichartz estimate to the factorized representation of the solution in the framework of \cite{CHPS}, one gets a better last cubic term $\||\phi|^2\phi\|_{L^{\frac{6}{5-2\epsilon}}}=\|\phi\|_{L^{\frac{18}{5-2\epsilon}}}^3$, and it allows us to reduce $s$ down to $\frac{2}{3}+\epsilon$. The regularity requirement in the classical Kato's work on the unconditional uniqueness for the 3D cubic NLS \cite{Kato95} can be covered in this way. We further push $s$ almost down to the critical regularity by 	employing negative order Sobolev norms (Lemma \ref{lemma: negLeibniz}), which are well-known tools in the literature on unconditional uniqueness for NLS. Combining the dispersive estimate, the Strichartz estimates and negative Sobolev norms, we formulate the key trilinear estimates (Lemma \ref{lemma: trilinear}) in our proof. 
\\

\noindent
\textbf{Organization of the paper}. We prove Theorem \ref{thm: main theorem} in Section \ref{sec: outline}, by reducing it to the main Lemma \ref{key estimate}. In Section \ref{sec: example}, we present an example calculation to explain the ingredients involved in the proof of Lemma \ref{key estimate}.  In Section \ref{sec: distinguished tree graph}, we introduce tree graphs for the organization of iterated Duhamel expansions, and give properties of the associated kernels. Finally, we prove the main Lemma \ref{key estimate} in Section \ref{sec: key lemma proof}. We prove the crucial trilinear estimates in Lemma \ref{lemma: trilinear} in Appendix \ref{sec: proof trilinear}.

\section{Proof of the Main Theorem}    \label{sec: outline}

In this section, we prove the main theorem. First, in \S\ref{sec: setup}, we present the setup of the proof. In \S\ref{sec: esy combinatorial method} we review Klainerman-Machedon's board game formulation \cite{KM} of the combinatorial argument of Erd\"os-Schlein-Yau \cite{ESY06, ESY07, ESY09, ESY10}. In \S\ref{subsection: prf thm w lemma}, we reduce the proof of the main theorem to the key lemma (Lemma \ref{key estimate}), via the quantum de Finetti theorem. The rest of the paper is then devoted to the proof of the lemma.

\subsection{Setup of the proof}  \label{sec: setup}
The setup of the proof is similar to that of Chen-Hainzl-Pavlovi\'c-Seiringer \cite{CHPS}, but we use a negative order Sobolev type norm to lower the regularity.\\ 

Let $\{\gamma_1^{(k)}(t)\}_{k\in\mathbb{N}}$ and $\{\gamma_2^{(k)}(t)\}_{k\in\mathbb{N}}$ be two mild solutions in $L_{t\in[0,T)}^{\infty}\mathfrak{H}^s$ to the cubic GP hierarchy with the same initial data, which are either admissible or limiting hierarchies. For uniqueness, it is enough to show that their difference $\{\gamma^{(k)}(t)\}_{k\in\mathbb{N}}$, given by
$$ \gamma^{(k)}(t):=\gamma_1^{(k)}(t)-\gamma_2^{(k)}(t),  \qquad k\in\mathbb{N},$$
vanishes for all $k$ in a certain norm. 

Due to the linearity of the GP hierarchy, it follows that the difference $\{\gamma^{(k)}(t)\}_{k\in\mathbb{N}}$ solves the GP hierarchy with zero initial data. Hence, each $\gamma^{(k)}(t)$ satisfies the integral equation
$$\gamma^{(k)}(t)=i\lambda \int_0^t U^{(k)}(t-t_1)B_{k+1}\gamma^{(k+1)}(t_1)dt_1.$$
Now fix $k$. Iterating this integral equation $(n-1)$ times, we write
$$ \gamma^{(k)}(t)=(i\lambda)^n \int_{t_n\leq \cdots \leq t_1\leq t} U^{(k)}(t-t_1)B_{k+1} \cdots U^{(k+n-1)}(t_{n-1}-t_n)B_{k+n}\gamma^{(k+n)}(t_n)dt_1\cdots dt_n.$$
For notational convenience, we denote $(k+1)$-temporal variables $(t_0,t_1,\cdots,t_n)$ by $\underline{t}_n$ with 
$t_0=t$, and the linear propagator $U^{(i)}(t_j-t_{j'})$ by $U_{j,j'}^{(i)}$. Then, we rewrite $\gamma^{(k)}(t)$ in a compact form as
\begin{equation}   \label{rewrite Duhamel n fold}
 \gamma^{(k)}(t)= (i\lambda)^n\int_{t_n\leq \cdots \leq t_1\leq t} J^k(\underline{t}_n)d\underline{t}_n,
\end{equation}
where
$$J^{k}(\underline{t}_n):=U_{0,1}^{(k)}B_{k+1}U_{1,2}^{(k)}B_{k+2} \cdots U_{n-1,n}^{(k+n-1)}B_{k+n}\gamma^{(k+n)}(t_n).$$
By density, our uniqueness theorem follows from uniqueness in an even weaker norm.
\begin{prop}  \label{prop: zero trace norm}
For all $t\in[0, T)$ with $T>0$ small enough, the trace norm of $S^{(k,-d)}\eqref{rewrite Duhamel n fold}$ vanishes as $n \to \infty$ uniformly in $k$, that is
\begin{equation}  \label{zero trace norm}
{\rm Tr}(|S^{(k,-d)}\gamma^{(k)}(t)|)=0, \quad \forall k,
\end{equation}
where $d>0$ is the dimension.
\end{prop}
 
\subsection{Erd\"os-Schlein-Yau Combinatorial method in board-game form}  \label{sec: esy combinatorial method}
One obstacle in showing uniqueness is the number of terms in $J^k(\underline{t}_n)$. Indeed, each $B_{k+i}$ is a sum of $(k+i-1)$ terms. Thus, in the expansion of $J^k(\underline{t}_n)$, there are a total of $k(k+1)\cdots(k+n-1)=\mathcal{O}(n!)$ terms for fixed $k$. We solve this problem by using the powerful combinatorial methods of Erd\"os-Schlein-Yau \cite{ESY06, ESY07, ESY09, ESY10} in the board-game formulation of Klainerman-Machedon \cite{KM}. 

The key idea of the \emph{board game} arguments is that, by grouping the large number of integral terms into equivalence classes in which we have control, we can avoid estimating the rapidly increasing number of terms one by one. Throughout this section, we present a few lemmas that will help us group these terms and derive bounds on certain equivalence classes.

Let $\mu$ be a map from $\{k+1,k+2,\cdots,k+n\}$ to $\{1,2,\cdots,k+n-1\}$ such that $\mu(2)=1$ and $\mu(j)<j$ for all $j$. Denotes by $\mathcal{M}_{k,n}$ the set of all such maps. 

We express the operators $B_{k+i}$ and $J^k$ in terms of map $\mu$.  We have
$$B_{k+i}=\sum_{j=1}^{k+i-1} B_{j;k+i}=\sum_{\mu\in \mathcal{M}_{k,n}} B_{\mu(k+i);k+i}$$
and
\begin{equation}   \label{Jk mapform}
 J^k(\underline{t}_n)=\sum_{\mu \in \mathcal{M}_{k,n}} J^k(\underline{t}_n;\mu),
\end{equation}
where 
$$J^k(\underline{t}_n;\mu)=U^{(k)}(t-t_1)B_{\mu(k+1);k+1}U^{(k+1)}(t_1-t_2) \cdots U^{(k+n-1)}(t_{n-1}-t_n)B_{\mu(k+n); k+n}\gamma^{(k+n)}(t_n).$$

By the definition of $\mu$, we can represent $\mu$ by highlighting exactly one nonzero entry $B_{\mu(k+l),k+l}$ ($l$-th column, $\mu(k+l)$-th row) in each column of a $(k+n-1)\times n$ matrix. Since $\mu(k+l)<k+l$, we set the remaining entries of the matrix equal to $0$. 
\begin{equation}
 \begin{pmatrix}
  \mathbf{B_{1;k+1}}&B_{1;k+2} & \cdots &\mathbf{B_{1;k+n}} \\
  B_{2;k+1} & B_{2;k+2} & \cdots &B_{2;k+n}  \\
  \cdots & \cdots &  \cdots & \cdots  \\
  B_{k;k+1}&\mathbf{B_{k;k+2}}& \cdots &B_{k;k+n}  \\
  0 & B_{k+1;k+2} & \cdots &B_{k+1;k+n}  \\
  \cdots & \cdots & \cdots  & \cdots  \\
  0 & 0 & \cdots & B_{k+n-1;k+n}  \\
 \end{pmatrix}
\end{equation}

Henceforth, we can rewrite \eqref{rewrite Duhamel n fold} as
\begin{equation}
  \gamma^{(k)}(t)=\int_0^{t}\cdots \int_0^{t_{n}}\sum\limits_{\mu \in \mathcal{M}_{k,n}} J^k(\underline{t}_{k+n};\mu)dt_1 \dots dt_n.
\end{equation}

Here the time domain $\{t_n\leq t_{n-1}\leq \cdots \leq t\} \subset [0,t]^n$ is the same for all $\mu\in \mathcal{M}_{k,n}$. We now consider the terms $ I(\mu,\sigma)$ in the sum $\gamma^{(k)}(t)=\sum_{\mu \in \mathcal{M}_{k,n}} I(\mu,\sigma)$.  We have
 
\begin{equation}  \label{basic integral term}
 I(\mu,\sigma)=\int_{t_{\sigma(n)} \leq t_{\sigma(n-1)} \leq\cdots \leq t}  J^k(\underline{t}_{k+n};\mu)dt_1 \dots dt_n,
\end{equation}
where $\sigma$ is a permutation of ${1,2,\dots,n}$. We associate the integral $I(\mu,\sigma)$ the following $(k+n) \times n$ matrix. We may also use it to visualize $B_{\mu(k+j);k+j}$ that correspond to a highlighted entry. 
\begin{equation}
  \begin{pmatrix}   \label{integral matrix}
  t_{\sigma^{-1}(1)} & t_{\sigma^{-1}(2)}  & \cdots & t_{\sigma^{-1}(n)}  \\
  \mathbf{B_{1;k+1}}&B_{1;k+2} & \cdots &\mathbf{B_{1;k+n}} \\
  B_{2;k+1} & B_{2;k+2}  & \cdots &B_{2;k+n}  \\
  \cdots & \cdots &\cdots & \cdots  \\
  B_{k;k+1}&\mathbf{B_{k;k+2}}  &\cdots &B_{k;k+n}  \\
  0 & B_{k+1;k+2}  &\cdots &B_{k+1;k+n}  \\
  \cdots & \cdots  &\cdots  & \cdots  \\
  0 & 0 &  \cdots & B_{k+n}  \\
 \end{pmatrix}
\end{equation}
The columns of matrix \eqref{integral matrix} are labeled $1$ through $n$, and the rows are labeled $0$ through $k+n-1$. 

Each term \eqref{basic integral term} corresponds to a unique matrix of form \eqref{integral matrix}. A key observation is that two matrices of this form can have to the same value for $I(\mu;\sigma)$ given that one matrix can be transformed to another under the so called \emph{acceptable moves}. 

In the following paragraph, we will present a few key lemmas to help us with the combinatorial reduction. For the proof of these lemmas, we refer the reader to \cite{ESY06, ESY07, ESY09, ESY10, KM, CPquintic, UniqueGP}. 

\subsubsection{Acceptable Moves} If $\mu(k+j+1)<\mu(k+j)$, we take the following steps at the same time
\begin{itemize}
 \item exchange the highlights in columns $j$ and $j+1$
 \item exchange the highlights in rows $k+j$ and $k+j+1$
 \item exchange $t_{\sigma^{-1}(j)}$ and $t_{\sigma^{-1}(j+1)}$
\end{itemize}
The exchange only happens when there is a highlight, if there is no highlight we will skip that step. The following lemma highlights the necessity to introduce \emph{equivalence classes}. 

\begin{lemma}  \label{lem:equal integral}
 Let $(\mu,\sigma)$ be transformed into $(\mu',\sigma')$ by an acceptable move. Then, for the corresponding integrals \eqref{basic integral term}, we have $I(\mu,\sigma)=I(\mu',\sigma')$
\end{lemma}

\subsubsection{Equivalence Class}
 Consider the subset $\{\mu_s\} \subset \mathcal{M}_{k,n}$ of \emph{special upper echelon} matrices in which each highlighted element of a higher row is to the left of each highlighted element of a lower row. An example of a special upper echelon matrix (with $k=1,n=4$) is
\begin{equation*}
  \begin{pmatrix}
  \mathbf{B_{1;2}} & \mathbf{B_{1;3}} & B_{1;4} & B_{1;5} \\
  0 & B_{2;3} & B_{2;4} & B_{2;5}  \\
  0 & 0 & \mathbf{B_{3;4}} & B_{3;5}  \\
  0 & 0 & 0 & \mathbf{B_{4;5}}  \\
 \end{pmatrix}
\end{equation*}

\begin{lemma}
 For each element of $\mathcal{M}_{k,n}$ there is a finite number of acceptable moves which brings the matrix to upper echelon form. 
\end{lemma}
\iffalse
\begin{proof}
We start from the first row and take acceptable moves to bring all highlighted entries in the first row in consecutive order. Since our goal is the upper echelon form, the updated highlighted entries will occupy $\bf{B_{1;k+1}}$ through $\bf{B_{1;k+j_1-1}}$. Then if there are any highlighted entries on the second row, bring them to positions $\bf{B_{2;k+j_1}}$ through $\bf{B_{2;k+j_2-1}}$ for some $j_2>j_1$. Noticed that this will not effect the highlighted positions of the first row. If there is no highlighted entire on the second row, just leave it and move to the third row. Keep repeating these steps and we will end up with a special upper echelon matrix after finitely many steps.
\end{proof}
\fi

\begin{lemma}   \label{lem:number of echelon matrix}
 Let $C_{k,n}$ be the number of $(k+n-1) \times n$ special upper echelon matrices of the type discussed above. Then $C_{k,n} \leq 2^{k+2n-2}$.
\end{lemma}

\iffalse
\begin{proof}
The proof consists two steps. First of all, we dis-assemble the matrix by ``lifting'' all highlighted entries to the first row and put them in the same subsets if they were originally from the same row. In this way, the first row is partitioned into many subsets. Let $P_n$ denote the number of all possible partitions, then 
\begin{equation}    \label{P_n}
P_n=\sum\limits_{i=0}^{n-1}\binom{n-1}{i}=2^{n-1}
\end{equation}
The idea is to put $n-1$ pads in the space among the $n$ elements to separate them. Since we can separate them into different numbers (from $1$ to $n$) of subsets, we can choose to use $0$ pads, $1$ pads, $\cdots$, upto $n-1$ pads. Hence \eqref{P_n} follows. 

The second step is to re-assemble the upper echelon matrix by ``lowering'' the first subset to the first used row, the second subset to the second used row, etc. Note here, we do not require that only the upper triangle matrix is used, which may result in more matrices. This does not matter since we are looking for an upper bound of the number of such matrices. Suppose an arbitrary partition of $n$ has $i$ subsets. Then there will be exactly $\binom{k+n-1}{i}$ ways to lower them in an order preserving way to the $k+n-1$ available rows. 
Thus
\begin{equation*}
C_{k,n} \leq P_n \sum\limits_{i=0}^{n} \binom{k+n-1}{i} \leq 2^{k+n-1+n-1} \leq 2^{k+2n-2}
\end{equation*}
as desired.
\end{proof}
\fi

Let $\mu_s$ be a special upper echelon matrix. We say $\mu$ is in the equivalence class of $\mu_s$: $\mu \sim \mu_s$ if $\mu$ can be transformed to $\mu_s$ in finitely many acceptable moves.

\begin{thm}    \label{thm: equal domain integral}
 There exists a subset $D$ of $[0,t]^n$ such that
\small
\begin{equation}
\sum\limits_{\mu \sim \mu_s}\int_0^{t} ... \int_0^{t_{n-1}} J^k(\underline{t}_{n};\mu)dt_{1} \dots dt_{n} = \int\limits_{D} J^k(\underline{t}_{n};\mu)dt_{1} \dots dt_{n}.
\end{equation}
\end{thm}

\begin{proof}
We perform finitely many acceptable moves on the matrix associated to the integral
\begin{equation*}
I(\mu,id)=\int_0^{t} ... \int_0^{t_{n-1}} J^k(\underline{t}_{n};\mu)dt_{1} \dots dt_{n}.
\end{equation*}
Let $I(\mu,id)$ be the integral associated to the upper echelon matrix obtained. By Lemma \ref{lem:equal integral}
\begin{equation*}
I(\mu,id)=I(\mu_s,\sigma).
\end{equation*}
Assume that $(\mu_1,id)$ and $(\mu_2,id)$ with $\mu_1 \neq \mu_2$ yield the same echelon form $\mu_s$.  Then the corresponding permutations $\sigma_1$ and $\sigma_2$ must be different. Therefore, $D$ can be chosen to be the union of all $\{t \geq t_{\sigma(1)} \geq t_{\sigma(2)} \geq \cdots \geq  t_{\sigma(n)} \}$ for all permutations $\sigma$ which occur in a given equivalence class of some $\mu_s$.
\end{proof}

With the above theorem, we are able to reduce the sum of $\mathcal{O}(n!)$ terms into a sum of $\mathcal{O}(C^n)$ terms:
\begin{equation}  \label{reduced Duhamel n fold}
  \gamma^{(k)}(t)=\sum_{\sigma\in \mathcal{M}_{k,n}} \int_{D_{\sigma,t}} d\underline{t}_nJ^k(\underline{t}_n;\sigma),
\end{equation}
which we can afford.

\subsection{Proof of the main theorem} \label{subsection: prf thm w lemma}
As mentioned above, it suffices to show Proposition \ref{prop: zero trace norm}. For the proof, we uses the framework of Chen-Hainzl-Pavlovi\'c-Seiringer \cite{CHPS} via the quantum de Finetti theorem.\\

Applying the strong or the weak quantum de Finetti theorem, we write 
\begin{equation}  \label{de finetti form}
\gamma^{(k)}(t)=\int d\tilde{\mu}_t(\phi) (\Ket{\phi} \Bra{\phi})^{\otimes{k}}, \quad \forall k\in\mathbb{N},
\end{equation}
where $\tilde{\mu}_t=\mu_t^{(1)}-\mu_t^{(2)}$ with
$$\gamma_i^{(k)}(t)=\int d\mu_t^{(i)}(\phi) (\Ket{\phi} \Bra{\phi})^{\otimes{k}},\quad i=1,2.$$
Plugging \eqref{de finetti form} into $J^k(\underline{t}_n;\sigma)$ in the reduced Duhamel expansion $(\ref{reduced Duhamel n fold})$, we obtain a new expression
\begin{equation}  \label{reduced de Finetti Duhamel n fold}
  \gamma^{(k)}(t)=\sum_{\sigma\in \mathcal{M}_{k,n}} \int_{D_{\sigma,t}} d\underline{t}_n \int d\tilde{\mu}_{t_n}(\phi) J^k(\underline{t}_n;\sigma),
\end{equation}
where
\begin{equation}   \label{Jk before factorization}
 \begin{split}
  J^k(\underline{t}_n;\sigma)=U_{0,1}^{(k)}B_{\sigma(k+1);k+1}U_{1,2}^{(k+1)}B_{\sigma(k+2);k+2}\cdots U_{n-1,n}^{(k+n-1)}B_{\sigma(k+n); k+n} (\Ket{\phi} \Bra{\phi})^{\otimes{(k+n)}}.
 \end{split}
\end{equation}
Then, we formulate the following key lemma that implies Proposition \ref{prop: zero trace norm} (and thus the main theorem).
\begin{lemma}[Key lemma]\label{key estimate}
There exists a uniform constant $C>0$ such that for arbitrarily small $\epsilon>0$, we have
\begin{align}\label{multi-d}
&\int_{[0,T)^{n-1}} d\underline{t}_{n-1} \textup{Tr}(|S^{(k,-d)}J^k(\underline{t}_n;\sigma)|)\leq
\begin{cases}
(CT^\epsilon)^{n-1}\|\phi\|_{H^{s_{\epsilon}}}^{2(k+n)}&\text{ if } d\ge 3\\
(CT^{1/3})^{n-1}\|\phi\|_{H^{1/3}}^{2(k+n)}&\text{ if } d=2\\
(CT^{1/2})^{n-1}\|\phi\|_{H^{1/6}}^{2(k+n)}&\text{ if }d=1,
\end{cases}
\end{align}
where $s_{\epsilon}=\frac{d-2}{2}+\epsilon$.
\end{lemma}

\begin{proof}[Proof of Theorem \ref{thm: main theorem}, assuming Lemma \ref{key estimate}]
We present the proof for the case $d\geq 3$ only. Indeed, when $d=1$ ($d=2$, resp), it can be proved in an analogous way by  replacing the $H^{s_c}$ norm with the $H^{1/6}$ norm (the $H^{1/3}$ norm, resp).

Let $\{\gamma^{(k)}(t)\}_{k\in\mathbb{N}}$ be as above. The goal is to show that $\textup{Tr}(|S^{(k, -d)}\gamma^{(k)}(t)|)=0$ for all $k\in\mathbb{N}$. Applying the triangle inequality and Lemma \ref{key estimate}, we write
\begin{equation}
\begin{aligned}
\textup{Tr}(|S^{(k,-d)}\gamma^{(k)}(t)|)&\leq\sum_{i=1,2}\sum_{\sigma\in \mathcal{M}_{k,n}} \int_{[0,T)^n} d\underline{t}_n \int d\mu^{(i)}_{t_n}(\phi) \textup{Tr}(|S^{(k,-d)}J^k(\underline{t}_n;\sigma)|)\\
&\leq (CT^\epsilon)^{n-1}T \sum_{i=1,2}\sum_{\sigma\in \mathcal{M}_{k,n}}\sup_{t_n\in[0,T)}\int d\mu_{t_n}^{(i)}(\phi) \|\phi\|_{H^{s_{\epsilon}}}^{2(k+n)}.
\end{aligned}
\end{equation}
We claim that there exists $M>0$ such that 
\begin{equation}\label{unif bdd in measure}
\|\phi\|_{H^{s_{\epsilon}}}\leq M\quad\textup{a.s. }\mu_t^{(i)},\quad\forall t\in[0,T).
\end{equation} 
Indeed, since $\{\gamma^{(k)}(t)\}_{k\in\mathbb{N}}\in L_{t\in[0,T)}^\infty\mathfrak{H}^s$, there exists $M>0$ such that
\begin{equation}
\int d\mu_t^{(i)}(\phi) \|\phi\|_{H^s}^{2k}=\textup{Tr}(|S^{(k,s)} \gamma^{(k)}(t)|)<M^{2k}, \quad \forall k\in \mathbb{N}.
\end{equation}
Hence, it follows from the Chebyshev inequality that for $\lambda>M$,
\begin{equation}
\mu_t^{(i)} \big(\{\phi\in L^2: \|\phi\|_{H^s}>\lambda \}\big) \leq \frac{1}{\lambda^{2k}} \int d\mu_t^{(i)}(\phi) \|\phi\|_{H^s}^{2k}<\Big(\frac{M}{\lambda}\Big)^{2k}\to 0\quad \textup{ as }k\to\infty. 
\end{equation}
Returning to $(2.14)$, by $(\ref{unif bdd in measure})$ and Lemma \ref{lem:number of echelon matrix}, we prove that 
\begin{equation}
{\rm Tr}(|S^{(k,-d)}\gamma^{(k)}(t)|)\le (CT^\epsilon)^{n-1}T \cdot2\cdot2^{k+2n-2}\cdot M^{2(k+n)}=\frac{M^{2k}2^{k-1}T}{CT^{\epsilon}}(4CT^{\epsilon}M^2)^n\rightarrow 0\text{ as }n\rightarrow \infty.
\end{equation}
for $T<(4CM^2)^{-1/\epsilon}$. 
\end{proof}

The remainder of our paper will be devoted to proving Lemma \ref{key estimate}. We remark that our proof heavily relies on the following trilinear estimates which combine the dispersive estimate, the Strichartz estimates and negative Sobolev norms. The proof of these trilinear estimates is given in the appendix.
\begin{lemma}[Trilinear estimates]\label{lemma: trilinear}
We define the trilinear form $T$ by
$$T(f,g,h)=(e^{i(t-t_1)\Delta}f) (e^{i(t-t_2)\Delta}g) (e^{i(t-t_3)\Delta}h).$$
$(i)$ $d\geq 3$. For small $\epsilon>0$, we have
\begin{align}
\|T(f,g,h)\|_{L_{t\in[0,T)}^1W_x^{-(s_c+\frac{\epsilon}{2}),r_{\epsilon}}}&\lesssim T^\epsilon\|f\|_{W^{-(s_c+\frac{\epsilon}{2}),r_{\epsilon}}}\|g\|_{H^{s_{\epsilon}}}\|h\|_{H^{s_{\epsilon}}},\label{L^p bound}\\
\|T(f,g,h)\|_{L_{t\in[0,T)}^1H_x^{s_{\epsilon}}}&\lesssim T^\epsilon\|f\|_{H^{s_{\epsilon}}} \|g\|_{H^{s_{\epsilon}}}\|h\|_{H^{s_{\epsilon}}}, \label{H^s bound} 
\end{align}
where $s_{\epsilon}=s_c+\epsilon=\frac{d-2}{2}+\epsilon$, $r_{\epsilon}=\frac{2d}{d+2(1-\epsilon)}$.\\
$(ii)$ $d=2$. For small $\epsilon>0$, we have
\begin{align} 
\|T(f,g,h)\|_{L_{t\in[0,T)}^1W_x^{-(\frac{1}{3}-\frac{\epsilon}{2}), \frac{2}{2-\epsilon}}}&\lesssim T^\epsilon\|f\|_{W^{-(\frac{1}{3}-\frac{\epsilon}{2}),\frac{2}{2-\epsilon}}}\|g\|_{H^{1/3}}\|h\|_{H^{1/3}}, \label{2d bound} \\
\|T(f,g,h)\|_{L_{t\in[0,T)}^1H_x^{1/3}}&\lesssim T^{1/3}\|f\|_{H^{1/3}} \|g\|_{H^{1/3}}\|h\|_{H^{1/3}} \label{2d H13 bound}.
\end{align}
$(ii)$ $d=1$. We have
\begin{align}
\|T(f,g,h)\|_{L_{t\in[0,T)}^1L_x^1}&\lesssim T^{1/2}\|f\|_{L^1}\|g\|_{L^2}\|h\|_{L^2},\label{L^1 bound}\\
\|T(f,g,h)\|_{L_{t\in[0,T)}^1L_x^2}&\lesssim T^{1/2}\|f\|_{L^2} \|g\|_{L^2}\|h\|_{L^2}\label{L^2 bound}.
\end{align}
\end{lemma}

We will prove Lemma \ref{key estimate} in the following sections. To this end, we will proceed as in \cite{CHPS} and use \emph{binary tree graphs} to help organize the terms in $J^k(\underline{t}_n,\sigma)$ (see $(\ref{Jk before factorization})$). For the reader's convenience, before proving the lemma, we give an example calculation in Section \ref{sec: example}. We remark that the trilinear estimates in Lemma \ref{lemma: trilinear} are the key estimates, and will be applied recursively in general case (see Section \ref{sec: key lemma proof}).

\section{An Example}    \label{sec: example}

In this section, we illustrate the ideas of the proof of Lemma \ref{key estimate} via an example.\\

Let $d\geq 3$, $k=2$ and $n=4$ in Lemma \ref{key estimate}. We investigate the example
\begin{equation}\label{Jk example Duhamel}
\int_{[0,T)^3} d\underline{t}_3 \textup{Tr}(|S^{(2,-d)}J^2(\underline{t}_4;\sigma)|)\end{equation}
with a specific map $\sigma$ represented by the matrix 
\begin{equation}  \label{Jk example matrix}
 \begin{pmatrix}
  \mathbf{B_{1;3}} & B_{1;4} & B_{1;5} & B_{1;6} \\
  B_{2;3} & \mathbf{B_{2;4}} & B_{2;5} & B_{2,6}  \\
  0 & B_{3;4} & \mathbf{B_{3;5}} & \mathbf{B_{3,6}} \\
  0 & 0 & B_{4;5} & B_{4,6}\\
 \end{pmatrix}.
\end{equation}
In other words,
\begin{equation}  \label{Jk example}
 J^2=J^2(\underline{t}_4;\sigma)=U_{0,1}^{(2)} B_{1,3}U_{1,2}^{(3)} B_{2,4}U_{2,3}^{(4)} B_{3,5}U_{3,4}^{(5)} B_{3,6} (\Ket{\phi} \Bra{\phi})^{\otimes 6}.
\end{equation}
To this end, in \S 3.1-3.2, we organize the terms in $J^2(\underline{t}_4,\sigma)$. Then, in \S3.3, we estimate the example by the trilinear estimates (Lemma \ref{lemma: trilinear}).

\subsection{Factorization of $J^2$} \label{subsection: factorization}
We will decompose $J^2$ into two one-particle density matrices by examining the effect of the contraction operators starting with the last one on the RHS of \eqref{Jk example}. We denote each factor in the last term $(\Ket{\phi} \Bra{\phi})^{\otimes 6}$ by $u_i$, ordered by increasing index $i$, so that $(\Ket{\phi} \Bra{\phi})^{\otimes 6}=\otimes_{i=1}^6 u_i$.\\

First of all, in \eqref{Jk example}, the last interaction operator $B_{3,6}$ contracts the factor $u_3$ and $u_6$, and leaves all other factors unchanged,
\begin{equation}   \label{B36 action}
B_{3,6} (\otimes_{i=1}^6 u_i)=u_1\otimes u_2\otimes \Theta_4\otimes u_4\otimes u_5.
\end{equation}
where  $$\Theta_4:=B_{1,2}(u_3\otimes u_6).$$
The index $\alpha$ in $\Theta_{\alpha}$ associates $\Theta_{\alpha}$ to the $\alpha$-th interaction operator from the left in \eqref{Jk example}. Since we only run the expansion to the $n$-th level, we have $1\leq \alpha\leq n$. In this specific case, $n=4$, the $4$th interaction operator is $B_{3,6}$. 

Next, $B_{3,5}$ contracts $U_{3,4}^{(1)}\Theta_4$ and $U_{3,4}^{(1)}u_5$,
\begin{equation}   \label{B35 action}
 B_{3,5}U_{3,4}^{(5)}(\eqref{B36 action})=(U_{3,4}^{(2)}(u_1\otimes u_2)) \otimes \Theta_3 \otimes (U_{3,4}^{(1)}u_4),
\end{equation}
where $$\Theta_3:=B_{1,2}((U_{3,4}^{(1)}\Theta_4)\otimes (U_{3,4}^{(1)}u_5)).$$
Then, by the semigroup property, $U_{2,3}^{(i)} U_{3,4}^{(i)}=U_{2,4}^{(i)}$. The operator $B_{2,4}$ contracts $U_{2,4}^{(1)}u_2$ with $U_{2,4}^{(1)}u_4$, which correspond to the 2nd and 5th factors in \eqref{B35 action}.  The other factors are left invariant.
\begin{equation}   \label{B24 action}
 B_{2,4}U_{2,3}^{(4)} (\eqref{B35 action})=(U_{2,4}^{(1)}u_1)\otimes \Theta_2 \otimes (U_{2,3}^{(1)}\Theta_3),
\end{equation}
where $$\Theta_2=B_{1,2}(U_{2,4}^{(2)}(u_2\otimes u_4)).$$
Finally, $B_{1,3}$ contracts $(U_{1,4}^{(1)}u_1)$ and $(U_{1,3}^{(1)}\Theta_3)$ and leaves other factors unchanged.
\begin{equation}   \label{B13 action}
 B_{1,3}U_{1,2}^{(3)} (\eqref{B24 action})=\Theta_1 \otimes (U_{1,2}^{(1)}\Theta_2),
\end{equation}
where $$\Theta_1=B_{1,2}((U_{1,4}^{(1)}u_1)\otimes(U_{1,3}^{(1)}\Theta_3)).$$
Therefore, $J^2$ can be factorized as
\begin{equation}   \label{factorized Jk example}
 J^2=(U_{0,1}^{(1)}\Theta_1) \otimes (U_{0,2}^{(1)}\Theta_2):=J^1_1 \otimes J^1_2.
\end{equation}
In the above expression we may write the factors $J^1_j$ (for $j\leq k=2$) as one-particle matrices and substitute with $u_i=\Ket{\phi}\Bra{\phi}$, for $i\leq k+n=6$. Thus, it follows that
\begin{equation}   \label{relabel J11}
J^1_1=U_{0,1}^{(1)} B_{1,2} U_{1,3}^{(2)} B_{2,3} U_{3,4}^{(3)} B_{2,4}(\Ket{\phi}\Bra{\phi})^{\otimes 4}
\end{equation}
where we relabel the index in operators $B_{\sigma_1(r),r}$ such that the interaction operators in \eqref{relabel J11} correspond to $B_{1,3}, B_{3,5}, B_{3,6}$ respectively, and most importantly keep the connectivity structure between them. The relabeling function $\sigma_1$ (see the notation in \eqref{Jk before factorization}) take values: $\sigma_1(2)=1, \sigma_1(3)=2, \sigma_1(4)=3$. Moreover, for $j=1$, we perform the relabeling in the same spirit find that
\begin{equation}   \label{relabel J12}
 J^1_2=U_{0,2}^{(1)} B_{1,2} U_{2,4}^{(2)} (\Ket{\phi}\Bra{\phi})^{\otimes 2}
\end{equation}
where $\sigma_2(2)=1$.

We note that for any $l<l'$, the interaction operators $B_{\sigma(l),l}$ and $B_{\sigma(l'),l'} $ in $J^2$ (associated to the matrix \eqref{Jk example matrix}) belong to the same factor $J^1_j$ if either $\sigma(l)=\sigma(l')$ or $\sigma(l')=l$. In such cases, we consider them as being \emph{connected}. This connectivity structure is exactly the key point of the Duhamel terms that we want to illustrate using binary tree graphs. Each $\sigma_j$ can be viewed as the restriction of $\sigma$ to $J^1_j$. We call factors that have a free propagator applied to each $\phi$ (like $J^1_2$) \emph{regular} and factors that involve the contractions of $(\Ket{\phi}\Bra{\phi})^{\otimes 2}$ without free propagator in between (like $J^1_1$) \emph{distinguished}.

\subsection{Recursive determination of contraction structure} \label{subsection: contraction structure}
Next, repeating the argument in \S\ref{subsection: factorization}, we express the kernel of each factor explicitly.\\

Consider the distinguished factor $J_1^1$.  For $\alpha=1,2,3$, we denote by $\Theta_\alpha$ the kernel obtained after contracting a two particle density matrix to a one particle matrix via the interaction operator. We will determine $\Theta_\alpha$ recursively in the normal form
\begin{equation}  \label{general kernel}
\Theta_\alpha(x,x')=\sum_{\beta_\alpha} c_{\beta_\alpha}^\alpha \psi_{\beta_\alpha}^\alpha(x) \overline{\chi_{\beta_\alpha}^\alpha}(x'),\ c_{\beta_\alpha}^\alpha=\pm 1
\end{equation}
from the last interaction operator. First, contracting variables by $B_{2,4}$, we get
\begin{equation}  \label{contract by B24}
B_{2,4}(|\phi\ra\la\phi|)^{\otimes 4}=(|\phi\ra\la\phi|)\otimes \Theta_3\otimes (|\phi\ra\la\phi|)
\end{equation}
with
$$\Theta_3(x,x')=|\phi|^2\phi(x) \overline{\phi}(x')-\phi(x)\overline{|\phi|^2\phi}(x')=\sum_{\beta_3=1}^2c_{\beta_3}^3\psi_{\beta_3}^3(x)\overline{\chi_{\beta_3}^3}(x').$$
Next, contracting variables by $B_{2,3}$,
\begin{equation}  \label{contract by B23}
B_{2,3}U_{3,4}^{(3)}\eqref{contract by B24}=(|U_{3,4}\phi\ra\la U_{3,4}\phi|)\otimes \Theta_2,
\end{equation}
where $U_{i,j}:=e^{i(t_i-t_j)\Delta}$ and
\begin{align*}
\Theta_2(x,x')&=\sum_{\beta_3=1}^2c_{\beta_3}^3\Big(U_{3,4}\psi_{\beta_3}^3|U_{3,4}\phi|^2\Big)(x) \overline{U_{3,4}\chi_{\beta_3}^3}(x')-c_{\beta_3}^3U_{3,4}\psi_{\beta_3}^3(x)\Big(\overline{U_{3,4}\psi_{\beta_3}^3}|U_{3,4}\chi|^2\Big)(x')\\
&=:\sum_{\beta_2=1}^4  c_{\beta_2}^2\psi_{\beta_2}^2(x)\overline{\chi_{\beta_2}^2}(x').
\end{align*}
Finally, by the first interaction operator $B_{1,2}$,
$$B_{1,2}U_{1,3}^{(2)}\eqref{contract by B23}=B_{1,2}\Big(|U_{1,4}\phi\ra\la U_{1,4}\phi|\otimes \sum_{\beta_2=1}^4  c_{\beta_2}^2|U_{1,3}\psi_{\beta_2}^2\ra\la U_{1,3}\chi_{\beta_2}^2|\Big)=\Theta_1,$$
where $\Theta_1(x,x')$ is given by
\begin{align*}
&\sum_{\beta_2=1}^4c_{\beta_2}^2 \Big(U_{1,4}\phi U_{1,3}\psi_{\beta_2}^2\overline{U_{1,3}\chi_{\beta_2}^2}\Big)(x) \overline{U_{1,4}\phi}(x')-c_{\beta_2}^2U_{1,4}\phi(x) \Big(\overline{U_{1,4}\phi U_{1,3}\psi_{\beta_2}^2}U_{1,3}\chi_{\beta_2}^2\Big)(x')\\
&=:\sum_{\beta_1=1}^8  c_{\beta_1}^1\psi_{\beta_1}^1(x)\overline{\chi_{\beta_1}^1}(x').
\end{align*}
Therefore, $J_1^1$ can be represented by
$$J_1^1(x,x')=U_{0,1}^{(1)}\Theta_1(x,x')=\sum_{\beta_1=1}^8  c_{\beta_1}^1U_{0,1}\psi_{\beta_1}^1(x)\overline{U_{0,1}\chi_{\beta_1}^1}(x'),$$

Similarly, we write the regular factor $J_2^1$ as
$$J_2^1(\sigma_2; t_2, t_4)=U_{0,1}^{(1)}\tilde{\Theta}_1(x,x')=\sum_{\tilde{\beta}_1=1}^2  \tilde{c}_{\tilde{\beta}_1}^1U_{0,1}\tilde{\psi}_{\tilde{\beta}_1}^1(x)\overline{U_{0,1}\tilde{\chi}_{\tilde{\beta}_1}^1}(x'),$$
where 
\begin{align*}
\tilde{\Theta}_1(x,x')&=(|U_{2,4}\phi|^2U_{2,4}\phi)(x) \overline{U_{2,4}\phi}(x')-U_{2,4}\phi(x)(|U_{2,4}\phi|^2\overline{U_{2,4}\phi})(x')\\
&=:\sum_{\tilde{\beta}_1=1}^2\tilde{c}_{\tilde{\beta}_1}^1\tilde{\psi}_{\tilde{\beta}_1}^1(x)\overline{\tilde{\chi}_{\tilde{\beta}_1}^1}(x').
\end{align*}

\subsection{Recursive Estimates} \label{subsection: example proof of lemma key estimate}

Now, we estimate the example $(\ref{Jk example Duhamel})$ using the structural properties obtained from the previous two subsections. The key tool is the trilinear estimates (Lemma \ref{lemma: trilinear}).
\\

Observe that in the example $(\ref{Jk example Duhamel})$, the distinguished factor $J_1^1$ is independent of $t_2$, and the regular factor $J_2^1$ depends only on $t_2$ and $t_4$ (see \eqref{relabel J11} and \eqref{relabel J12}). Thus, \eqref{Jk example Duhamel} can be factored as
\begin{equation}
\eqref{Jk example Duhamel}=\Big(\int_{[0,T)^2} dt_1 dt_3 \textup{Tr}(|S^{(1,-d)}J_1^1|) \Big)\Big(\int_0^T dt_2 \textup{Tr}(|S^{(1,-d)}J_2^1|)\Big).
\end{equation}
We estimate these two factors separately.
\subsubsection{Distinguished factor}
By \S \ref{subsection: factorization} and \S\ref{subsection: contraction structure}, we have
\begin{equation}\label{Distinguished example2}
\int_{[0,T)^2} dt_1 dt_3 \textup{Tr}(|S^{(1,-d)}J_1^1|)\leq\sum_{\beta_1=1}^8\int_{[0,T)^2} dt_1 dt_3 \|\psi_{\beta_1}^1\|_{H^{-d}}\|\chi_{\beta_1}^1\|_{H^{-d}},
\end{equation}
where for each $\beta_\alpha$, only one out of two terms $\psi_{\beta_\alpha}^\alpha$ and $\chi_{\beta_\alpha}^\alpha$ is cubic. Among the eight integrals on the right hand side of $(\ref{Distinguished example2})$, we estimate the following two cases.\\ \\
\textbf{Case 1.} Consider the integral whose $\psi_{\beta_\alpha}^{\alpha}$'s are all cubic, precisely
\begin{equation}   \label{all cubic case}
\begin{aligned}
&\psi_{\beta_1}^1=U_{1,4}\phi U_{1,3}\psi_{\beta_2}^2\overline{U_{1,3}\chi_{\beta_2}^2}, &&\chi_{\beta_1}^1=U_{1,4}\phi,\\
&\psi_{\beta_2}^2=U_{3,4}\psi_{\beta_3}^3|U_{3,4}\phi|^2, &&\chi_{\beta_2}^2=U_{3,4}\chi_{\beta_3}^3,\\
&\psi_{\beta_3}^3=|\phi|^2\phi,&&\chi_{\beta_3}^3=\phi.
\end{aligned}
\end{equation}
We apply the trilinear estimates \eqref{L^p bound} recursively keeping the $W^{-s_c+\frac{\epsilon}{2}, r_{\epsilon}}$ norm on $\psi_{\beta_\alpha}^\alpha$. Then, we obtain that
\begin{align*}  
\int_{[0,T)^2} dt_1 dt_3 \|\psi_{\beta_1}^1\|_{H^{-d}}\|\chi_{\beta_1}^1\|_{H^{-d}}&\lesssim\int_{[0,T)^2}dt_1dt_3\|\psi_{\beta_1}^1\|_{W^{-(s_c+\frac{\epsilon}{2}), r_{\epsilon}}}\|\chi_{\beta_1}^1\|_{H^{s_\epsilon}} \quad\textup{(by Sobolev ineq)}\\
&=\int_{[0,T)^2}dt_1dt_3\|U_{1,4}\phi U_{1,3}\psi_{\beta_2}^2\overline{U_{1,3}\chi_{\beta_2}^2}\|_{W^{-(s_c+\frac{\epsilon}{2}), r_{\epsilon}}}\|\phi\|_{H^{s_\epsilon}}\\
&\leq C_0T^\epsilon\int_0^T dt_3\|\psi_{\beta_2}^2\|_{W^{-(s_c+\frac{\epsilon}{2}), r_{\epsilon}}}\|\chi_{\beta_2}^2\|_{H^{s_{\epsilon}}}\|\phi\|_{H^{s_{\epsilon}}}^2\quad\textup{(by \eqref{L^p bound})}\\
&=C_0T^\epsilon\int_0^Tdt_3 \|U_{3,4}\psi_{\beta_3}^3|U_{3,4}\phi|^2\|_{W^{-(s_c+\frac{\epsilon}{2}),r_{\epsilon}}}\|\phi\|_{H^{s_{\epsilon}}}^3\\
&\leq (C_0T^{\epsilon})^2\|\psi_{\beta_3}^3\|_{W^{-(s_c+\frac{\epsilon}{2}), r_{\epsilon}}}\|\phi\|_{H^{s_{\epsilon}}}^5\quad\textup{(by \eqref{L^p bound})}\\
&=(C_0T^{\epsilon})^2\||\phi|^2\phi\|_{W^{-(s_c+\frac{\epsilon}{2}), r_{\epsilon}}}\|\phi\|_{H^{s_{\epsilon}}}^5\\
&\lesssim (C_0T^{\epsilon})^2\|\phi\|_{H^{s_{\epsilon}}}^8\quad\textup{(by Sobolev ineq)}.
\end{align*}
\textbf{Case 2.} Consider the integral whose $\psi_{\beta_\alpha}^{\alpha}$'s are all linear except the last one, that is,
\begin{equation}  \label{mixed case}
\begin{aligned}
  &\psi_{\beta_1}^1=U_{1,3}\psi_{\beta_2}^2, &&\chi_{\beta_1}^1=U_{1,3}\chi_{\beta_2}^2|U_{1,4}\phi|^2,\\
  &\psi_{\beta_2}^2=U_{3,4}\psi_{\beta_3}^3, &&\chi_{\beta_2}^2=U_{3,4}\chi_{\beta_3}^3|U_{3,4}\phi|^2,\\
  &\psi_{\beta_3}^3=|\phi|^2\phi,&&\chi_{\beta_3}^3=\phi.
\end{aligned}
\end{equation}
In this case, we first combine linear propagators acting on $\psi_{\beta_3}^3$ so that
$$\psi_{\beta_1}^1=U_{1,3}U_{3,4}(|\phi|^2\phi)=U_{1,4}(|\phi|^2\phi).$$
Then, applying the trilinear estimate \eqref{H^s bound} twice, we obtain
\begin{align*}
\int_{[0,T)^2}dt_1dt_3\|\psi_{\beta_1}^1\|_{H^{-d}}\|\chi_{\beta_1}^1\|_{H^{-d}} &\lesssim \int_{[0,T)^2}dt_1dt_3\|U_{1,4}(|\phi|^2\phi)\|_{H^{-d}} \|U_{1,3}\chi_{\beta_2}^2 |U_{1,4}\phi|^2\|_{H^{s_{\epsilon}}}\\
&= \int_{[0,T)^2}dt_1dt_3\||\phi|^2\phi\|_{H^{-d}} \|U_{1,3}\chi_{\beta_2}^2 |U_{1,4}\phi|^2\|_{H^{s_{\epsilon}}} \\
&\leq C_0T^{\epsilon}\int_0^T dt_3 \||\phi|^2\phi\|_{W^{-(s_c+\frac{\epsilon}{2}), r_{\epsilon}}} \|\chi_{\beta_2}^2\|_{H^{s_{\epsilon}}} \|\phi\|^2_{H^{s_{\epsilon}}} \quad\textup{(by \eqref{H^s bound})}\\
&\leq (C_0T^{\epsilon})^2\||\phi|^2\phi\|_{W^{-(s_c+\frac{\epsilon}{2}), r_{\epsilon}}} \|\phi\|^5_{H^{s_{\epsilon}}} \quad\textup{(by \eqref{H^s bound})}\\
&\lesssim (C_0T^{\epsilon})^2 \|\phi\|_{H^{s_{\epsilon}}}^8\quad\textup{(by Sobolev ineq)},
\end{align*}
which is the same bound as in Example 1.\\

Similarly, one can show that the other six integrals satisfy the same bound. Then, it follows that
$$\int_{[0,T)^2} dt_1 dt_3 \textup{Tr}(|S^{(1,-d)}J_1^1|)\lesssim 8(C_0T^{\epsilon})^2 \|\phi\|_{H^{s_{\epsilon}}}^8.$$

\subsubsection{Regular factor} 
For the regular factor, we have
\begin{equation}\label{Regular example2}
\int_0^T dt_2 \textup{Tr}(|S^{(1,-d)}J_2^1|)\leq\sum_{\tilde{\beta}_1=1}^2\int_0^T dt_2 \|\tilde{\psi}_{\tilde{\beta}_1}^1\|_{H^{-d}}\|\tilde{\chi}_{\tilde{\beta}_1}^1\|_{H^{-d}},
\end{equation}
where for each $\tilde{\beta}_1$, only one out of two terms $\tilde{\psi}_{\tilde{\beta}_1}^1$ and $\tilde{\chi}_{\tilde{\beta}_1}^1$ is cubic. For instance, when $\tilde{\psi}_{\tilde{\beta}_1}^1=|U_{2,4}\phi|^2 U_{2,4}\phi$ and $\tilde{\chi}_{\tilde{\beta}_1}^1=U_{2,4}\phi$, it follows from the trilinear estimate \eqref{H^s bound} that
$$\int_0^Tdt_2\|\tilde{\psi}_{\tilde{\beta}_1}^1\|_{H^{-d}}\|\tilde{\chi}_{\tilde{\beta}_1}^1\|_{H^{-d}}\leq\int_0^Tdt_2\||U_{2,4}\phi|^2U_{2,4}\phi\|_{H^{s_{\epsilon}}}\|U_{2,4}\phi\|_{H^{s_{\epsilon}}}\leq C_0T^\epsilon \|\phi\|_{H^{s_{\epsilon}}}^4.$$
Similarly, one can also show that the other integral satisfies the same bound. Therefore, we get
$$\int_0^T dt_2 \textup{Tr}(|S^{(1,-d)}J_2^1|)\leq 2C_0T^{\epsilon}\|\phi\|_{H^{s_{\epsilon}}}^4$$

\subsubsection{Conclusion} Going back to $(3.14)$), we conclude that
$$\eqref{Jk example Duhamel}\lesssim 2^4 \cdot(C_0T^\epsilon)^3\|\phi\|_{H^{s_{\epsilon}}}^{12}.$$

\section{Binary tree graphs for the general case}   \label{sec: distinguished tree graph}

In order to prove Lemma \ref{key estimate} in the general case, we proceed as in \cite{CHPS}, and use binary tree graphs.  These graphs will help us keep track of the contraction operations applied iteratively in the Duhamel expansion \eqref{reduced de Finetti Duhamel n fold}.

\subsection{The binary tree graphs}   \label{subsec: binary tree graph}
We begin by recalling that, by \eqref{Jk before factorization}, $J^k$ is given by
\begin{align*}
  J^k(\underline{t}_n;\sigma)=U_{0,1}^{(k)}B_{\sigma(k+1);k+1}U_{1,2}^{(k+1)}B_{\sigma(k+2);k+2}\cdots U_{n-1,n}^{(k+n-1)}B_{\sigma(k+n); k+n} (\Ket{\phi} \Bra{\phi})^{\otimes{(k+n)}},
\end{align*}
where
\begin{align*}
(\pphi)^{\otimes (k+n)}(\underline{x}_{k+n};\underline{x}_{k+n}')=\prod_{i=1}^{k+n}(\pphi)(x_i;x_i')
\end{align*}
is a product of one-particle kernels.  Since the free evolution operators $U$ and the contraction operators $B$ preserve the product structure, it follows that we can also decompose
\begin{align}
J^k(t,t_1,\dots,t_r;\sigma;\underline{x}_{k};\underline{x}_{k}')=\prod_{j=1}^k J^1_j(t,t_{\ell_{j,1}},\dots,t_{l_{j,m_j}};\sigma_j;x_{j};x_{j}')\label{J decomposition}
\end{align}
into a product of one-particle kernels $J^1_j$.  We associate to this decomposition $k$ disjoint binary tree graphs $\tau_1, \tau_2, \dots, \tau_k$. These graphs appear as \emph{skeleton graphs} in \cite{ESY06, ESY07, ESY09, ESY10}. As in \cite{CHPS}, we assign \emph{root, internal,} and \emph{leaf} vertices to for each tree $\tau_j$.
\begin{itemize}
 \item A \emph{root} vertex labeled as $W_j$, $j=1,2,\cdots,k$, to represent $J^1_j(x_j,x'_j)$.
 \item An \emph{internal} vertex labeled by $v_l$, $l=1,2,\cdots,n$, corresponding to $B_{\sigma(k+l),k+l}$ and attached to the time variable $t_l$.
 \item A \emph{leaf} vertex $u_i$, $i=1,2,\cdots,k+n$, representing each factor $(\Ket{\phi}\Bra{\phi})(x_i,x'_i)$.
\end{itemize}

Next, we connect the vertices with \emph{edges}, as described below.
\begin{itemize}
 \item If $v_l$ is the smallest value of $l$ such that $\sigma(k+l)=j$, then we connect $v_l$ to the root vertex $W_j$ and write $W_j\sim v_l$ (or equivalently $W_j\sim B_{\sigma(k+l),k+l}$).  If there is no internal vertex connected to a root vertex $W_j$, then we connect $W_j$ to the leaf $u_j$, and write $W_j \sim u_j$.
 \item For any $1<l\leq n$, if $\exists l'>l$ such that $\sigma(k+l)=\sigma(k+l')$ or $\sigma(k+l')=k+l$, then we connect $v_{l}$ and $v_{l'}$ and write $v_{l} \sim v_{l'}$ (or equivalently $B_{\sigma(k+l),k+l}\sim B_{\sigma(k+l'),k+l'}$). In this case, we call $v_l$ the \emph{parent vertex} of $v_{l'}$, and $v_{l'}$ the \emph{child vertex} of $v_l$. We denote the two child vertices of $v_l$ by $v_{k_{-}(l)}$ and $v_{k_{+}(l)}$, with $k_{-}(l)<k_{+}(l)$. 
 
 \item When there is no internal vertex with $r'>r$ and $k+l=\sigma(k+l')$, we connect $v_{l}$ to the leaf vertex $u_{k+l}$ and write $v_{l} \sim u_{k+l}$ (or equivalently $B_{\sigma(k+l),k+l}\sim u_{k+l}$).  If there is no internal vertex with $l'>l$ and $\sigma(k+l)=\sigma(k+l')$, then we connect $v_l$ to the leaf vertex $u_{\sigma(k+l)}$ and write $v_l \sim u_{\sigma(k+l)}$ (or equivalently $B_{\sigma(k+l),k+l}\sim u_{\sigma(k+l)}$). 
\end{itemize}

We remark that it follows from the construction above that each root vertex has only one child vertex, and each internal vertex has exactly two child vertices (which can be internal and leaf). We call the tree $\tau_j$ \emph{distinguished}  if $v_n \in \tau_j$, and \emph{regular} if $v_n\notin \tau_j$. The two leaves connected to $v_n$ are called \emph{distinguished leaf vertices}, and all other leaves are called \emph{regular leaf vertices}.  Clearly, there are $k-1$ regular trees and one distinguished tree in each binary tree graph. 

A sample binary tree graph is given in Figure \ref{Figure binary_tree}, for $J^k$ as in \eqref{Jk example}. Each tree $\tau_j$ has root vertex $W_j$, for $j=1,2$. The two leaf vertices $u_3$ and $u_6$ and the internal vertex $v_4$ (or $B_{3,6}$) are distinguished. $\tau_1$ is the distinguished tree, and is drawn with thick edges.
 
\begin{figure}
\centering
\def\svgwidth{3.5in}
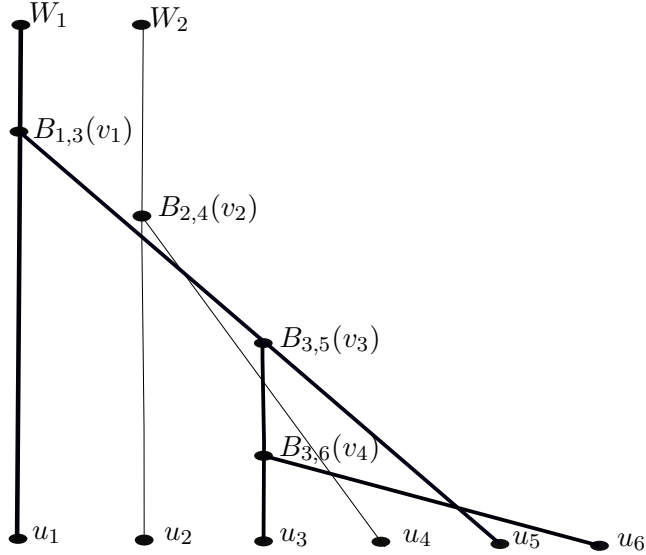
\caption{An example binary tree graphs of $J^k$. It is a disjoint union of two trees $\tau_1$ and $\tau_2$ with root vertices $W_1$ and $W_2$, respectively.  Each tree corresponds to a one-particle kernel in the example in section \ref{sec: example}, where $k=2$ and $n=4$.}\label{Figure binary_tree}
\end{figure}

\subsection{The distinguished one particle kernel $J^1_j$}
Let $\tau_j$ denote the distinguished tree graph.  It has $m_j$ internal vertices $(v_{\ell_j,\alpha})_{\alpha=1}^{m_j}$ and $m_j+1$ leaf vertices $(u_{j,i})_{i=1}^{m_j+1}$.  We enumerate the internal vertices with $\alpha\in\{1,\dots,m_j\}$ and the leaf vertices with $\alpha\in\{m_j+1,\dots,2m_j+2\}$.  To simplify notation, we refer to the vertex $v_{j,\alpha}$ by its label $\alpha$.  We observe that $J_j^1$ has the form
\begin{align}
&J_j^1(t,t_{\ell_{j,1}},\dots,t_{\ell_j,m_j};\sigma_j)\label{factorized J unsimplified}\\
&=U^{(1)}(t-t_{\ell_{j,1}})\cdots U^{(1)}(t_{\ell_{j,1}-1}-t_{\ell_{j,1}})B_{\sigma_j(2),2}\cdots\nonumber\\
&\hspace{1cm}\cdots B_{\sigma_j(\alpha),\alpha}U^{(\alpha)}(t_{\ell_{j,\alpha-1}}-t_{\ell_{j,\alpha-1}+1})\cdots U^{(\alpha)}(t_{\ell_{j,\alpha}-1}-t_{\ell_{j,\alpha}})B_{\sigma_j(\alpha+1),\alpha+1}\cdots\nonumber\\
&\hspace{1cm}\cdots U^{(m_j)}(t_{\ell_j,m_j-1}-t_{l_j,m_j})B_{\sigma_j(m_j+1),m_j+1}(\pphi)^{\otimes(m_j+1)}.\nonumber
\end{align}
By the group property
\begin{align*}
U^{(\alpha)}(t)U^{(\alpha)}(s)=U^{(\alpha)}(t+s),
\end{align*}
and the fact that $\sigma_j(2)=1$, \eqref{factorized J unsimplified} reduces to
\begin{align}
&J_j^1(t,t_{\ell_{j,1}},\dots,t_{\ell_j,m_j};\sigma_j)\label{factorized J}\\
&=U^{(1)}(t-t_{\ell_{j,1}})B_{1,2}\cdots\nonumber\\
&\hspace{1cm}\cdots B_{\sigma_j(\alpha),\alpha}U^{(\alpha)}(t_{\ell_{j,\alpha-1}}-t_{\ell_{j,\alpha}})B_{\sigma_j(\alpha+1),\alpha+1}\cdots\nonumber\\
&\hspace{1cm}\cdots U^{(m_j)}(t_{\ell_j,m_j-1}-t_{l_j,m_j})B_{\sigma_j(m_j+1),m_j+1}(\pphi)^{\otimes(m_j+1)},\nonumber
\end{align}
where $\ell_{j,m_j}=r$.

\subsection{Definition of the kernels $\Theta_\alpha$ at the vertices of the distinguished tree graph}

In this section, we proceed as in \cite{CHPS}, and recursively assign a kernel $\Theta_\alpha$ to each vertex $\alpha$ of the distinguished tree graph.  The kernels at the vertices of the regular tree graph are defined similarly.  We begin by assigning the kernel
\begin{align*}
\Theta_\alpha(x;x'):=\phi(x)\overline{\phi(x')}
\end{align*}
to the leave vertex with label $\alpha\in\{m_j+1,\dots,2m+j+2\}$ (corresponding to $u_{j,\alpha-m_j}$).

Next, we determine $\Theta_{m_j}$ at the distinguished vertex $\alpha=m_j$ from the term on the last line of \eqref{factorized J}, given by
\begin{align*}
B_{\sigma_j(m_j+1),m_j+1}(\pphi)^{\otimes (m_j+1)}
&=(\pphi)^{\otimes(\sigma_j(m_j+1)-1)}\otimes\Theta_{m_j}\\
&\hspace{1cm}\otimes(\pphi)^{\otimes(m_j+1-\sigma_j(m_j+1)-1)}
\end{align*}
where
\begin{align}
\Theta_{m_j}(x;x'):=\tilde{\psi}(x)\overline{\phi(x')}-\phi(x)\overline{\tilde{\psi(x')}}\label{theta_mj}
\end{align}
with $\tilde{\psi}:=|\phi|^2\phi$.  It is obtained from contracting two copies of $\pphi$ at the two leaf vertices $\kappa_-(m_j),\kappa_+(m_j)$ which have $m_j$ as their parent vertex.

Now we are ready to begin the induction.  Let $\alpha\in\{1,\dots,m_j-1\}$.  Suppose that the kernels $\Theta_{\alpha'}$ have been determined for all $\alpha'>\alpha$.  We let $\kappa_-(\alpha),\kappa_+(\alpha)$ label the two child vertices (of internal or leaf type) of $\alpha$,
\begin{align*}
\sigma_j(\alpha)=\sigma_j(\kappa_-(\alpha))\quad,\quad\alpha=\sigma_j(\kappa_+(\alpha)).
\end{align*}
Since $\Theta_{\kappa_-(\alpha)}$ and $\Theta_{\kappa_+(\alpha)}$ have already been determined, we can now define
\begin{align*}
&\Theta_{\alpha}(x;x')\\
&=B_{1,2}((U^{(1)}(t_\alpha-t_{\kappa_-(\alpha)})\otimes(U^{(1)}(t_\alpha-t_{\kappa_+(\alpha)}\Theta_{\kappa_+(\alpha)}))(x;x')\\
&=(U^{(1)}(t_\alpha-t_{\kappa_-(\alpha)})\Theta_{k_-(\alpha)})(x;x')[(U^{(1)}(t_\alpha-t_{\kappa_+(\alpha)})\Theta_{\kappa_+(\alpha)})(x;x)\\
&\hspace{2cm}-(U^{(1)}(t_\alpha-t_{\kappa_+(\alpha)})\Theta_{\kappa_+(\alpha)})(x';x')].
\end{align*}

The induction ends when we obtain the kernel $\Theta_1$ at $\alpha=1$.

\subsection{Key properties of the kernels $\Theta_\alpha$}\label{key properties}  As in \cite{CHPS}, we observe that the kernels $\Theta_\alpha$ satisfy the following properties.
\begin{itemize}
\item $\Theta_\alpha$ can be written as a sum of differences of factorized kernels
\begin{align}
\Theta_\alpha(x;x')=\sum_{\beta_\alpha}c_{\beta_\alpha}^\alpha\chi_{\beta_\alpha}^\alpha(x)\overline{\psi_{\beta_\alpha}^\alpha(x')}\label{factorized}
\end{align}
with at most $2^{m_j-\alpha}$ nonzero coefficients $c_{\beta_\alpha}^\alpha\in\{1,-1\}$.
\item The product $\chi_{\beta_\alpha}^\alpha(x)\overline{\psi_{\beta_\alpha}^\alpha(x')}$ in \eqref{factorized} above is either of the form
\begin{align}
\chi_{\beta_\alpha}^\alpha(x)\overline{\psi_{\beta_\alpha}^\alpha(x')}
&=(U_{\alpha;\kappa_-(\alpha)}\chi_{\beta_{\kappa_-(\alpha)}}^{\kappa_-(\alpha)})(x)\overline{(U_{\alpha;\kappa_-(\alpha)}\psi_{\beta_{\kappa_-(\alpha)}}^{\kappa_-(\alpha)})}(x')\nonumber\\
&\hspace{1cm}(U_{\alpha;\kappa_+(\alpha)}\chi_{\beta_{\kappa_+(\alpha)}}^{\kappa_+(\alpha)})(x)\overline{(U_{\alpha;\kappa_+(\alpha)}\psi_{\beta_{\kappa_+(\alpha)}}^{\kappa_+(\alpha)})}(x)\label{form1}
\end{align}
or
\begin{align}
\chi_{\beta_\alpha}^\alpha(x)\overline{\psi_{\beta_\alpha}^\alpha(x')}
&=(U_{\alpha;\kappa_-(\alpha)}\chi_{\beta_{\kappa_-(\alpha)}}^{\kappa_-(\alpha)})(x)\overline{(U_{\alpha;\kappa_-(\alpha)}\psi_{\beta_{\kappa_-(\alpha)}}^{\kappa_-(\alpha)})}(x')\nonumber\\
&\hspace{1cm}(U_{\alpha;\kappa_+(\alpha)}\chi_{\beta_{\kappa_+(\alpha)}}^{\kappa_+(\alpha)})(x')\overline{(U_{\alpha;\kappa_+(\alpha)}\psi_{\beta_{\kappa_+(\alpha)}}^{\kappa_+(\alpha)})}(x')\label{form2}
\end{align}
for some values of $\beta_{\kappa_-(\alpha)},\beta_{\kappa_+(\alpha)}$ that depend on $\beta_\alpha$.  Observe that above, the function $\chi_{\beta_\alpha}^\alpha$ is either of the cubic form
\begin{align}
\chi_{\beta_\alpha}^\alpha(x)&=(U_{\alpha;\kappa_-(\alpha)}\chi_{\beta_{\kappa_-(\alpha)}}^{\kappa_-(\alpha)})(x)\nonumber\\
&\hspace{1cm}(U_{\alpha;\kappa_+(\alpha)}\chi_{\beta_{\kappa_+(\alpha)}}^{\kappa_+(\alpha)})(x)\overline{(U_{\alpha;\kappa_+}(\alpha)\psi_{\beta_{\kappa_+}}^{\kappa_+(\alpha)})}(x)\label{cubic form}
\end{align}
or the linear form
\begin{align}
\chi_{\beta_\alpha}^\alpha(x)
&=(U_{\alpha;\kappa_-(\alpha)}\chi_{\beta_{\kappa_-(\alpha)}}^{\kappa_-(\alpha)})(x).\label{linear form}
\end{align}
Accordingly, $\psi_{\beta_\alpha}^\alpha$ respectively is either of linear or cubic form, and the product $\chi_{\beta_\alpha}^\alpha(x)\overline{\psi_{\beta_\alpha}^\alpha(x')}$ always has quartic form \eqref{form1} or \eqref{form2}.
\item We call the functions $\chi_{\beta_\alpha}^\alpha,\psi_{\beta_\alpha}^\alpha$ in the sum \eqref{factorized} \emph{distinguished} if they are a function of $|\phi|^2\phi$. In the product on the right hand side of \eqref{form1}, respectively \eqref{form2}, at most one of the four factors is distinguished.  Indeed, this is true for all regular leaf vertices, and for the distinguished vertex \eqref{theta_mj}.  By induction along decreasing values of $\alpha$, it is also true for the internal vertices.
\end{itemize}

\section{Proof of Lemma \ref{key estimate}}  \label{sec: key lemma proof}

In this section, we prove Lemma $\ref{key estimate}$.  We begin by considering the contribution of each factor $J_j^1$ on the right hand side of \eqref{J decomposition} separately.  One of these factors is distinguished, and will be dealt with in Proposition \ref{L^p integral} below.  Proposition \ref{H^s integral} will be for the regular factors.

We note that the analog of Proposition \ref{L^p integral} in \cite{CHPS} has a shorter proof.  This is because, where the authors of \cite{CHPS} work in $L^2$, we work in $W^{-(s_c+\tfrac{\epsilon}{2}),r_\epsilon}$ to achieve lower regularity.  In $W^{-(s_c+\tfrac{\epsilon}{2}),r_\epsilon}$, the linear propagators $e^{it\Delta}$ are no longer isometries, and so we have to carefully rearrange them so that they do not interfere with our proof.  This occurs in case 2 of our proof of Lemma \ref{induction step}.

We begin with Proposition \ref{L^p integral}, which addresses the contribution of the distinguished factor $J_j^1$.  We prove Proposition \ref{L^p integral} by induction.  Lemma \ref{first induction step} will serve as our first induction step, and Lemma \ref{induction step} will serve as the remainder of our proof by induction.

%To simplify notation, we denote the time variable $t_{\ell_j,\alpha}$ by $t_\alpha$.  We denote the subtree of $\tau_j$ with root at the vertex $\alpha$ by $\tau_{j,\alpha}$, and let
%\begin{align*}
%\int \bigg[\prod_{\alpha'\in\tau_{j,\alpha}}dt_{\alpha'}\bigg]:=\int_{[0,T)^{d_\alpha}}\bigg[\prod_{\alpha'\in\tau_{j,\alpha}}dt_{\alpha'}\bigg]
%\end{align*}
%be integration with respect to all time variables attached to the internal and root vertices of the subtree $\tau_{j,\alpha}$.  We let $d_\alpha$ denote the total number of internal and root vertices of $\tau_{j,\alpha}$.

\begin{prop}\label{L^p integral}
Let $d\ge 3$.  Then, for the distinguished tree $\tau_j$, we have the bound
\begin{align}
&\int_{[0,T)^{m_j-1}}dt_1\dots dt_{m_j-1}{\rm Tr}\bigg(\,\bigg|S^{(1,-d)}J^1_j(t,t_1,\cdots,t_{m_j};\sigma_j)\bigg|\,\bigg)\nonumber\\
&\hspace{3cm}\le 2^{m_j-1}C^{m_j-1}T^{\epsilon (m_j-1)}\|\phi\|_{H^{s_\epsilon}}^{2m_j-1}\||\phi|^2\phi\|_{W^{-(s_c+\tfrac{\epsilon}{2}),r_\epsilon}}.\label{distinguished bound d ge 2}
\end{align}
Similarly, when $d=2$, we have the bound
\begin{align}
&\int_{[0,T)^{m_j-1}}dt_1\dots dt_{m_j-1}{\rm Tr}\bigg(\,\bigg|S^{(1,-d)}J^1_j(t,t_1,\cdots,t_{m_j};\sigma_j)\bigg|\,\bigg)\nonumber\\
&\hspace{3cm}\le 2^{m_j-1}C^{m_j-1}T^{\tfrac{1}{3}(m_j-1)}\|\phi\|_{H^{1/3}}^{2m_j-1}\||\phi|^2\phi\|_{W^{-(\tfrac{1}{3}-\tfrac{\epsilon}{2}),r_\epsilon}},
\end{align}
and, when $d=1$, we have the bound
\begin{align}
&\int_{[0,T)^{m_j-1}}dt_1\dots dt_{m_j-1}{\rm Tr}\bigg(\,\bigg|S^{(1,-d)}J^1_j(t,t_1,\cdots,t_{m_j};\sigma_j)\bigg|\,\bigg)\nonumber\\
&\hspace{3cm}\le 2^{m_j-1}C^{m_j-1}T^{\tfrac{1}{2} (m_j-1)}\|\phi\|_{L^2}^{2m_j-1}\||\phi|^2\phi\|_{L^1}.\label{distinguished bound d=1}
\end{align}
\end{prop}
\begin{proof}
For $d\ge 3$, Proposition \ref{L^p integral} follows immediately from Lemma \ref{induction step} below.  Indeed, in the statement of Lemma \ref{induction step}, there are at most $2^{m_j-1}$ terms in the sum over $\beta_1$.

Observe that in the proofs of Lemmas \ref{first induction step} and \ref{induction step}, we use the bounds for $d\ge 3$ presented in Lemma \ref{lemma: trilinear}.  The proof of Proposition \ref{L^p integral} for $d=1,2$ is analogous (we use the corresponding bounds for $d=1,2$ presented in Lemma \ref{lemma: trilinear}).
\end{proof}

We now prove Lemma \ref{first induction step}, which will serve as the first induction step in our proof of Lemma \ref{L^p integral}.

\begin{lemma}\label{first induction step}
Let $d\ge 3$.  Then, the distinguished factor
\begin{align*}
J_j^1(\underline{t}_n;\sigma_j;x,x')=U^{(1)}(t-t_1)\sum_{\beta_1}c_{\beta_1}^1\psi_{\beta_1}^1(x)\chi_{\beta_1}^1(x')
\end{align*}
satisfies the following.  For each value of $\beta_1$, either there exits a non-negative integer $\ell<m_j-1$ such that
\begin{align}
&\int_{[0,T)^{m_j-1}}dt_1\dots dt_{m_j-1}{\rm Tr}\bigg(\,\bigg|S^{(1,-d)}U^{(1)}(t-t_1)c_{\beta_1}^1|\psi_{\beta_1}^1\ra\la\chi_{\beta_1}^1|\bigg|\,\bigg)\nonumber\\
&\le(CT^\epsilon)^{\ell}\sum_{\beta_1}\int_{[0,T)^{m_j-\ell-1}}dt_{\ell+1}\cdots dt_{m_j-1}\nonumber\\
&\hspace{2cm}\|(U_{\ell+2}f_{\ell+2}^1) (U_{\ell+2}f_{\ell+2}^2) (U_{\ell+2}f_{\ell+2}^3)\|_{W^{-s_c+\frac{\epsilon}{2},r_\epsilon}}\|U_{\ell+2}f_{\ell+2}^2\|_{H^{s_\epsilon}}\cdots\|U_{\ell+2}f_{\ell+2}^{2\ell+4}\|_{H^{s_\epsilon}},\label{dist}
\end{align}
where the functions $f$ are defined in terms of the functions $\psi_{\beta_\alpha}^\alpha$ and $\chi_{\beta_\alpha}^\alpha$ as described in the proof below, or
\begin{align}
&\int_{[0,T)^{m_j-1}}dt_1\dots dt_{m_j-1}{\rm Tr}\bigg(\,\bigg|S^{(1,-d)}U^{(1)}(t-t_1)c_{\beta_1}^1|\psi_{\beta_1}^1\ra\la\chi_{\beta_1}^1|\bigg|\,\bigg)\nonumber\\
&\le C^{m_j-1}T^{\epsilon (m_j-1)}\|\phi\|_{H^{s_\epsilon}}^{2m_j-1}\||\phi|^2\phi\|_{W^{-(s_c+\tfrac{\epsilon}{2}),r_\epsilon}}.\label{done early}
\end{align}
Moreover, $f_{\ell+2}^1$ is the only distinguished fuction on the right hand side of \eqref{dist}.
\begin{proof}
We recall that $U_{i,j}:=e^{i(t_i-t_j)\Delta}$, and let $U_j:=U_{j,j+1}$.  We have
\begin{align}
&\int_{[0,T)^{m_j-1}}dt_1\dots dt_{m_j-1}{\rm Tr}\bigg(\,\bigg|S^{(1,-d)}U^{(1)}(t-t_1)c_{\beta_1}^1|\psi_{\beta_1}^1\ra\la\chi_{\beta_1}^1|\bigg|\,\bigg)\nonumber\\
&\le\int_{[0,T)^{m_j-1}}dt_1\cdots dt_{m_j-1}\|\psi_{\beta_1}^1\|_{H^{-d}}\|\chi_{\beta_1}^1\|_{H^{-d}}\label{choice}.
%&\le\sum_{\beta_1}\int_{[0,T)^{m_j-1}}dt_1\cdots dt_{m_j-1}\|\psi_{\beta_1}^1\|_{W^{-(s_c+\tfrac{\epsilon}{2}),r_\epsilon}}\|\chi_{\beta_1}^1\|_{H^{s_\epsilon}}
%&\le\sum_{\beta_{\kappa_-(1)},\beta_{\kappa_+(1)}}CT^{\epsilon}\int_{[0,T)^{d_{\kappa_-(\alpha)}}}\bigg[\prod_{\alpha'\in\tau_{j,\kappa_-(\alpha)}}dt_{\alpha'}\bigg]\|\chi_{\beta_{\kappa_-(\alpha)}}^{\kappa_-(\alpha)}\|_{H^{s_\epsilon}}\|\psi_{\beta_{\kappa_-(\alpha)}}^{\kappa_-(\alpha)}\|_{H^{s_\epsilon}}\label{H^s step}\\
%&\hspace{1cm}\int_{[0,T)^{d_{\kappa_+(\alpha)}}}\bigg[\prod_{\alpha'\in\tau_{j,\kappa_+(\alpha)}}dt_{\alpha'}\bigg]\|\chi_{\beta_{\kappa_+(\alpha)}}^{\kappa_+(\alpha)}\|_{W^{-(s_c+\tfrac{\epsilon}{2}),r_\epsilon}}\|\psi_{\beta_{\kappa_+(\alpha)}}^{\kappa_+(\alpha)}\|_{H^{s_\epsilon}}\label{L^p step}
\end{align}
Now, we recall from subsection \ref{key properties} that one of functions $\psi_{\beta_1}^1,\chi_{\beta_1}^1$ is distinguished.  Moreover the distinguished function is either of the cubic form \eqref{cubic form} or of the linear form \eqref{linear form}.  We will now label the distinguished function $f_1^1$ and the regular function $f_1^2$.

\noindent\textbf{Case 1: $f_1^1$ is cubic.} If $f_1^1$ is cubic, then, by \eqref{form1} and \eqref{form2}, $f_1^1$ and $f_1^2$ are of the form
\begin{align*}
f_1^1&=(U_2f_2^1)(U_2f_2^2)(U_2f_2^3),\\
f_1^2&=U_2f_2^4.
\end{align*}
As in Section \ref{sec: example}, we apply the $W^{-s_c+\frac{\epsilon}{2},r_\epsilon}$ norm to the distinguished function $f_1^1$ and the $H^{s_\epsilon}$ norm to the regular function $f_1^2$ and find that
\begin{align}
\eqref{choice}&=\int_{[0,T)^{m_j-1}}dt_1\cdots dt_{m_j-1}\|f_1^1\|_{H^{-d}}\|f_1^2\|_{H^{-d}}\nonumber\\
&=\int_{[0,T)^{m_j-1}}dt_1\cdots dt_{m_j-1}\|(U_2f_2^1)(U_2f_2^2)(U_2f_2^3)\|_{H^{-d}}\|U_2f_2^4\|_{H^{-d}}\nonumber\\
&\le C\int_{[0,T)^{m_j-1}}dt_1\cdots dt_{m_j-1}\|(U_2f_2^1)(U_2f_2^2)(U_2f_2^3)\|_{W^{-(s_c+\tfrac{\epsilon}{2}),r_\epsilon}}\|U_2f_2^4\|_{H^{s_\epsilon}},\nonumber
\end{align}
which is of the form \eqref{dist}.

\noindent\textbf{Case 2: $f_1^2$ is cubic.}  In this case, we have that $f_1^1$ and $f_1^2$ are of the form
\begin{align*}
f_1^1&=U_2f_2^1,\\
f_1^2&=(U_2f_2^2)(U_2f_2^3)(U_2f_2^4).
\end{align*}
Since $f_1^1$ is distinguished, there exists $\ell\ge 1$ such that
$$f_{2}^1=U_{3}f_{3}^1,\ f_{3}^1=U_{4}f_{4}^1, ...,  f_{\ell}^1=U_{\ell+1}f_{\ell+1}^1, $$
and
\begin{equation}\label{distinguished cubic2}
f_{\ell+1}^1=(U_{\ell+2}f_{\ell+2}^1) (U_{\ell+2}f_{\ell+2}^2) (U_{\ell+2}f_{\ell+2}^3)\textup{ or }f_{\ell+1}^1=|\phi|^2\phi,
\end{equation}
where $f_{\ell+2}^1$ (or $f_{\ell+2}^2$ or $f_{\ell+2}^3$)  is a distinguished function.  Thus, combining all propagators acting on $f_{\ell+1}^1$, we write
\begin{align*}
f_1^1=U_{1,\ell+2}f_{\ell+1}^1.
\end{align*}
Again, we apply the $W^{-s_c+\frac{\epsilon}{2},r_\epsilon}$ norm to the distinguished function $f_1^1$ and the $H^{s_\epsilon}$ norm to the regular function $f_1^2$ and find that
\begin{align}
\eqref{choice}&=\int_{[0,T)^{m_j-1}}dt_1\cdots dt_{m_j-1}\|f_1^1\|_{H^{-d}}\|f_1^2\|_{H^{-d}}\nonumber\\
&=\int_{[0,T)^{m_j-1}}dt_1\cdots dt_{m_j-1}\|f_{\ell+1}^1\|_{H^{-d}}\|(U_2f_2^2)(U_2f_2^3)(U_2f_2^4)\|_{H^{-d}}\nonumber\\
&\lesssim\int_{[0,T)^{m_j-1}}dt_1\cdots dt_{m_j-1}\|f_{\ell+1}^1\|_{W^{-s_c+\frac{\epsilon}{2},r_\epsilon}}\|(U_2f_2^2)(U_2f_2^3)(U_2f_2^4)\|_{H^{s_\epsilon}}\label{first fs}.
\end{align}
Since $f_{\ell+1}$ doesn't depend on $t_1,\dots,t_\ell$, we find that after $\ell$ applications of \eqref{H^s bound},
\begin{align}
\eqref{first fs}\le (CT^\epsilon)^{\ell}\int_{[0,T)^{m_j-\ell-1}}dt_{\ell+1}\cdots dt_{m_j-1}\|f_{\ell+1}^1\|_{W^{-s_c+\frac{\epsilon}{2},r_\epsilon}}\|f_{\ell+1}^2\|_{H^{s_\epsilon}}\cdots\|f_{\ell+1}^{2\ell+4}\|_{H^{s_\epsilon}}.\label{ell fs}
\end{align}
If $f_{\ell+1}^1=|\phi|^2\phi$, then it follows from the binary tree graph structure presented in section \ref{sec: distinguished tree graph} that $\ell=m_j-1$ and $f_{\ell+1}^{\ell''}=\phi$ for $\ell''\ge 2$, and so we have proven \eqref{done early}.  Otherwise, if $f_{\ell+1}^1=(U_{\ell+2}f_{\ell+2}^1) (U_{\ell+2}f_{\ell+2}^2) (U_{\ell+2}f_{\ell+2}^3)$, then we have that
\begin{align*}
\eqref{ell fs}&\le(CT^\epsilon)^{\ell}\int_{[0,T)^{m_j-\ell-1}}dt_{\ell+1}\cdots dt_{m_j-1}\\
&\hspace{2cm}\|(U_{\ell+2}f_{\ell+2}^1) (U_{\ell+2}f_{\ell+2}^2) (U_{\ell+2}f_{\ell+2}^3)\|_{W^{-s_c+\frac{\epsilon}{2},r_\epsilon}}\|f_{\ell+1}^2\|_{H^{s_\epsilon}}\cdots\|f_{\ell+1}^{2\ell+4}\|_{H^{s_\epsilon}}\\
&=(CT^\epsilon)^{\ell}\int_{[0,T)^{m_j-\ell-1}}dt_{\ell+1}\cdots dt_{m_j-1}\\
&\hspace{2cm}\|(U_{\ell+2}f_{\ell+2}^1) (U_{\ell+2}f_{\ell+2}^2) (U_{\ell+2}f_{\ell+2}^3)\|_{W^{-s_c+\frac{\epsilon}{2},r_\epsilon}}\|U_{\ell+2}f_{\ell+2}^2\|_{H^{s_\epsilon}}\cdots\|U_{\ell+2}f_{\ell+2}^{2\ell+4}\|_{H^{s_\epsilon}},
\end{align*}
which is of the form \eqref{dist}.
\end{proof}
\end{lemma}

In Lemma \ref{induction step}, we complete the induction process.  Observe that in the proof below, we proceed as in the proof of Lemma \ref{first induction step}.  In each induction step, we apply the $W^{s_c+\frac{\epsilon}{2},r_\epsilon}$ norm to the distinguished function, and the $H^{s_\epsilon}$ norm to the regular functions.

\begin{lemma}\label{induction step}
Let $d\ge 3$.  Then, the distinguished factor
\begin{align*}
J_j^1(\underline{t}_n;\sigma_j;x,x')=U^{(1)}(t-t_1)\sum_{\beta_1}c_{\beta_1}^1\psi_{\beta_1}^1(x)\chi_{\beta_1}^1(x')
\end{align*}
satisfies the following.  For each value of $\beta_1$,
\begin{align}
&\int_{[0,T)^{m_j-1}}dt_1\dots dt_{m_j-1}{\rm Tr}\bigg(\,\bigg|S^{(1,-d)}U^{(1)}(t-t_1)c_{\beta_1}^1|\psi_{\beta_1}^1\ra\la\chi_{\beta_1}^1|\bigg|\,\bigg)\nonumber\\
&\le C^{m_j-1}T^{\epsilon (m_j-1)}\|\phi\|_{H^{s_\epsilon}}^{2m_j-1}\||\phi|^2\phi\|_{W^{-(s_c+\tfrac{\epsilon}{2}),r_\epsilon}}.\label{done on time}
\end{align}
\begin{proof}
By Lemma \ref{first induction step}, we have that for each $\beta_1$, either \eqref{done on time} holds, or there is a non-negative integer $\ell<m_j-1$ such that
\begin{align}
&\int_{[0,T)^{m_j-1}}dt_1\dots dt_{m_j-1}{\rm Tr}\bigg(\,\bigg|S^{(1,-d)}U^{(1)}(t-t_1)c_{\beta_1}^1|\psi_{\beta_1}^1\ra\la\chi_{\beta_1}^1|\bigg|\,\bigg)\nonumber\\
&\le(CT^\epsilon)^{\ell}2^{m_j-1}\int_{[0,T)^{m_j-\ell-1}}dt_{\ell+1}\cdots dt_{m_j-1}\nonumber\\
&\hspace{2cm}\|(U_{\ell+2}f_{\ell+2}^1) (U_{\ell+2}f_{\ell+2}^2) (U_{\ell+2}f_{\ell+2}^3)\|_{W^{-s_c+\frac{\epsilon}{2},r_\epsilon}}\|U_{\ell+2}f_{\ell+2}^2\|_{H^{s_\epsilon}}\cdots\|U_{\ell+2}f_{\ell+2}^{2\ell+4}\|_{H^{s_\epsilon}}\label{lth step},
\end{align}
where $f_{\ell+2}^1$ is the only distinguished function on the right hand side of \eqref{lth step}.  We recall from Section \ref{sec: distinguished tree graph} that $f_{\ell+2}^1$ is either of the cubic form \eqref{cubic form} or the linear for \eqref{linear form}.

Now, we will proceed by induction, and show that in each induction step, we can bound \ref{lth step} by an expression of the same form, but with a larger value of $\ell$.  In the last induction step, we find that \eqref{done} holds, which completes the proof of \eqref{done on time}.  Indeed, this follows from the binary tree graph structure presented in section \ref{sec: distinguished tree graph}.

\noindent\textbf{Case 1: $f_{\ell+2}^1$ is cubic.} If $f_{\ell+2}^1$ is cubic, then
\begin{align*}
f_{\ell+2}^1&=(U_{\ell+3}f_{\ell+3}^1) (U_{\ell+3}f_{\ell+3}^2) (U_{\ell+3}f_{\ell+3}^3),\\
f_{\ell+2}^2&=U_{\ell+3}f_{\ell+3}^4,\ f_{\ell+2}^3=U_{\ell+3}f_{\ell+3}^5,...,\ f_{\ell+2}^{2\ell+4}=U_{\ell+3}f_{\ell+3}^{2\ell+6}.
\end{align*}
Since $f_{\ell+2}^1$ is distinguished, one of $f_{\ell+3}^1, f_{\ell+3}^2, f_{\ell+3}^3$ is distinguished, say $f_{\ell+3}^1$. Then, applying $(\ref{L^p bound})$, we get the integral of the form $(\ref{lth step})$ back:
\begin{align*}
(\ref{lth step})&\lesssim (CT^{\epsilon})^{\ell+1}2^{m_j-1}\int_{[0,T)^{m_j-\ell-2}}dt_{\ell+2}\cdot\cdot\cdot dt_{m_j-1} \|f_{\ell+2}^1\|_{W^{-(s_c+\frac{\epsilon}{2}),r_\epsilon}}\|f_{\ell+2}^2\|_{H^{s_{\epsilon}}}\cdot\cdot\cdot \|f_{\ell+2}^{2\ell+4}\|_{H^{s_{\epsilon}}}.\\
&=(CT^{\epsilon})^{\ell+1}2^{m_j-1}\int_{[0,T)^{m_j-\ell-2}}dt_{\ell+2}\cdot\cdot\cdot dt_{m_j-1} \|(U_{\ell+3}f_{\ell+3}^1) (U_{\ell+3}f_{\ell+3}^2) (U_{\ell+3}f_{\ell+3}^3)\|_{W^{-(s_c+\frac{\epsilon}{2}),r_{\epsilon}}}\\
&\ \ \ \ \ \ \ \ \ \ \ \ \ \ \ \ \ \ \ \ \ \ \ \ \ \ \ \ \ \ \ \ \ \times\|f_{\ell+3}^4\|_{H^{s_{\epsilon}}}\cdot\cdot\cdot \|f_{\ell+3}^{2\ell+6}\|_{H^{s_{\epsilon}}}.
\end{align*}

\noindent\textbf{Case 2: $f_{\ell+2}^2$ is cubic.} If $f_{\ell+2}^1$ is cubic, then
\begin{align*}
f_{\ell+2}^1&=U_{\ell+3}f_{\ell+3}^1,\\
f_{\ell+2}^2&=(U_{\ell+3}f_{\ell+3}^2) (U_{\ell+3}f_{\ell+3}^3) (U_{\ell+3}f_{\ell+3}^4),\\
f_{\ell+2}^3&=U_{\ell+3}f_{\ell+3}^5,...,\ f_{\ell+2}^{2\ell+4}=U_{\ell+3}f_{\ell+3}^{2\ell+6}.
\end{align*}
Since $f_{\ell+2}^1$ is distinguished, there exists $\ell'\ge 1$ such that
\begin{align*}
f^1_{\ell+3}=U_{\ell+4}f^1_{\ell+4}, f^1_{\ell+4}=U_{\ell+5}f^1_{\ell+5},\dots,f^1_{\ell+1+\ell'}=U_{\ell+2+\ell'}f^1_{\ell+2+\ell'},
\end{align*}
and
\begin{align}
f^1_{\ell+2+\ell'}=(U_{\ell+3+\ell'}f_{\ell+3+\ell'}^1)(U_{\ell+3+\ell'}f_{\ell+3+\ell'}^2)(U_{\ell+3+\ell'}f_{\ell+3+\ell'}^3)\text{ or }f_{\ell+2+\ell'}^1=|\phi|^2\phi,\label{l' condition}
\end{align}
where $f^1_{\ell+3+\ell'}$ is a distinguished function.  Thus, combining all linear propagators acting on $f^1_{\ell+2+\ell'}$, we write
\begin{align*}
f_{\ell+2}^1=U_{\ell+2,\ell+3+\ell'}f_{\ell+2+\ell'}^1.
\end{align*}
Then, applying \eqref{L^p bound} and \eqref{H^s bound}, we obtain
\begin{align}
\eqref{lth step}&\le (CT^\epsilon)^{\ell+1}2^{m_j-1}\int_{[0,T)^{m_j-\ell-2}} dt_{\ell+2}\cdots dt_{m_j-1}\|f_{\ell+2+\ell'}^1\|_{W^{-(s_c+\frac{\epsilon}{2}),r_{\epsilon}}}\|f_{\ell+2}^2\|_{H^{s_\epsilon}}\cdots\|f_{\ell+2}^{2\ell+4}\|_{H^{s_\epsilon}}\nonumber\\
&\le (CT^\epsilon)^{\ell+2}2^{m_j-1}\int_{[0,T)^{m_j-\ell-3}} dt_{\ell+3}\cdots dt_{m_j-1}\|f_{\ell+2+\ell'}^1\|_{W^{-(s_c+\frac{\epsilon}{2}),r_{\epsilon}}}\|f_{\ell+3}^2\|_{H^{s_\epsilon}}\cdots\|f_{\ell+3}^{2\ell+6}\|_{H^{s_\epsilon}},\label{next step}
\end{align}
where, in the second inequality, we applied \eqref{H^s bound} to the cubic regular function $f_{\ell+2}^2$.  After $\ell'-1$ applications of \eqref{H^s bound}, we find that
\begin{align}
\eqref{next step}\le (CT^\epsilon)^{\ell+1+\ell'}2^{m_j-1}&\int_{[0,T)^{m_j-\ell-2-\ell'}}dt_{\ell+2+\ell'}\cdots dt_{m_j-1}\\
&\ \ \|f_{\ell+2+\ell'}^1\|_{W^{-(s_c+\frac{\epsilon}{2}),r_{\epsilon}}}\|f_{\ell+2+\ell'}^2\|_{H^{s_\epsilon}}\cdots\|f_{\ell+2+\ell'}^{2\ell+2\ell'+4}\|_{H^{s_\epsilon}}.\label{almost}
\end{align}
If
\begin{align}
f_{\ell+2+\ell'}^1=|\phi|^2\phi,\label{done}
\end{align}
then it follows from the binary tree graph structure presented in section \ref{sec: distinguished tree graph} that $\ell+2+\ell'=m_j$ and $f_{\ell+2+\ell'}^{\ell''}=\phi$ for $\ell''\ge 2$, and so we have completed the proof of \eqref{done on time}.  Otherwise, by \eqref{l' condition},
\begin{align*}
\eqref{almost}=(CT^\epsilon)^{\ell+1+\ell'}2^{m_j-1}&\int_{[0,T)^{m_j-\ell-2-\ell'}} dt_{\ell+2+\ell'}\cdots dt_{m_j-1}\\
&\ \ \|(U_{\ell+3+\ell'}f_{\ell+3+\ell'}^1)(U_{\ell+3+\ell'}f_{\ell+3+\ell'}^2)(U_{\ell+3+\ell'}f_{\ell+3+\ell'}^3)\|_{W^{-(s_c+\frac{\epsilon}{2}),r_{\epsilon}}}\\
&\times\|U_{\ell+3+\ell'}f_{\ell+3+\ell'}^2\|_{H^{s_\epsilon}}\cdots\|U_{\ell+3+\ell'}f_{\ell+3+\ell'}^{2\ell+2\ell'+4}\|_{H^{s_\epsilon}},
\end{align*}
which is of the form \eqref{lth step}.

\noindent\textbf{Case 3: $f_{\ell+2}^4$ is cubic.} This case can be treated like Case 2.  We choose $\ell'\ge 1$ satisfying \eqref{l' condition}, and combine linear propagators acting on $f_{\ell+2+\ell'}^1$.  Then, we repeat the above procedure to bound \eqref{lth step} by \eqref{next step}.
\end{proof}
\end{lemma}
Next, we consider the contribution of the regular factors $J_j^1$.
\begin{prop}\label{H^s integral}
Let $d\ge 3$.  Then, for the regular tree $\tau_j$, we have the bound
\begin{align}
&\int_{[0,T)^{m_j}}dt_1\dots dt_{m_j}{\rm Tr}\bigg(\,\bigg|S^{(1,-d)}J^1_j(t,t_1,\cdots,t_{m_j};\sigma_j)\bigg|\,\bigg)\nonumber\\
&\hspace{3cm}\le 2^{m_j}C^{m_j}T^{\epsilon m_j}\|\phi\|_{H^{s_\epsilon}}^{2m_j+2}.\label{regular bound d ge 2}
\end{align}
Similarly, when $d=2$, we have the bound
\begin{align}
&\int_{[0,T)^{m_j}}dt_1\dots dt_{m_j}{\rm Tr}\bigg(\,\bigg|S^{(1,-d)}J^1_j(t,t_1,\cdots,t_{m_j};\sigma_j)\bigg|\,\bigg)\nonumber\\
&\hspace{3cm}\le 2^{m_j}C^{m_j}T^{\tfrac{1}{3} m_j}\|\phi\|_{H^{1/3}}^{2m_j+2},
\end{align}
and, when $d=1$, we have the bound
\begin{align}
&\int_{[0,T)^{m_j}}dt_1\dots dt_{m_j}{\rm Tr}\bigg(\,\bigg|S^{(1,-d)}J^1_j(t,t_1,\cdots,t_{m_j};\sigma_j)\bigg|\,\bigg)\nonumber\\
&\hspace{3cm}\le 2^{m_j}C^{m_j}T^{\tfrac{1}{2}m_j}\|\phi\|_{L^2}^{2m_j+2}.\label{regular bound d=1}
\end{align}
\begin{proof}
Again, we consider the case $d\ge 3$, and note that the proof for $d=1,2$ is analogous (based on using the bounds for $d=1,2$ in Lemma \ref{lemma: trilinear}).

We now proceed with the proof for $d\ge 3$.
\begin{align}
&\int_{[0,T)^{m_j}}dt_1\dots dt_{m_j}{\rm Tr}\bigg(\,\bigg|S^{(1,-d)}J^1_j(t,t_1,\cdots,t_{m_j};\sigma_j)\bigg|\,\bigg)\nonumber\\
&=\int_{[0,T)^{m_j}}dt_1\dots dt_{m_j}{\rm Tr}\bigg(\,\bigg|S^{(1,-d)}U^{(1)}(t-t_1)\Theta_1\bigg|\,\bigg)\nonumber\\
&\le\sum_{\beta_1}\int_{[0,T)^{m_j}}dt_1\cdots dt_{m_j}\|\psi_{\beta_1}^1\|_{H^{-d}}\|\chi_{\beta_1}^1\|_{H^{-d}}\nonumber\\
&\le\sum_{\beta_1}\int_{[0,T)^{m_j}}dt_1\cdots dt_{m_j}\|\psi_{\beta_1}^1\|_{H^{s_\epsilon}}\|\chi_{\beta_1}^1\|_{H^{s_\epsilon}}\label{choice2}
\end{align}
By \eqref{form1} and \eqref{form2}, one of $\psi_{\beta_1}^1,\chi_{\beta_1}^1$ is cubic, and the other is linear.  We define $f_1^1$ to be the cubic function, and $f_1^2$ to be the linear one.  Then, by \eqref{form1} and \eqref{form2}, $f_1^1$ and $f_1^2$ are of the form
\begin{align*}
f_1^1&=(U_2f_2^1)(U_2f_2^2)(U_2f_2^3).\\
f_1^2&=U_2f_2^4.
\end{align*}
By \eqref{H^s bound}, we have
\begin{align}
\eqref{choice2}&=\sum_{\beta_1}\int_{[0,T)^{m_j}}dt_1\cdots dt_{m_j}\|(U_2f_2^1)(U_2f_2^2)(U_2f_2^3)\|_{H^{s_\epsilon}}\|U_2f_2^4\|_{H^{s_\epsilon}}\label{regular form}\\
&\le(CT^\epsilon)\sum_{\beta_1}\int_{[0,T)^{m_j-1}}dt_2\cdots dt_{m_j}\|f_2^1\|_{H^{s_\epsilon}}\|f_2^2\|_{H^{s_\epsilon}}\|f_2^3\|_{H^{s_\epsilon}}\|f_2^4\|_{H^{s_\epsilon}}.\label{regular step 1}
\end{align}
By construction, only one of the factors $f_2^\ell$ is cubic.  Without loss of generality, $f_2^1$ is cubic, and so we have
\begin{align*}
f_2^1&=(U_3f_3^1)(U_3f_3^2)(U_3f_3^3),\\
f_2^\ell&=U_3f_3^{\ell+2}\quad\text{ for }\ell=2,3,4.
\end{align*}
Thus,
\begin{align*}
\eqref{regular step 1}=(CT^\epsilon)\sum_{\beta_1}\int_{[0,T)^{m_j-1}}dt_2\cdots dt_{m_j}\|(U_3f_3^1)(U_3f_3^2)(U_3f_3^3)\|_{H^{s_\epsilon}}\|U_3f_3^4\|_{H^{s_\epsilon}}\|U_3f_3^5\|_{H^{s_\epsilon}}\|U_3f_3^6\|_{H^{s_\epsilon}},
\end{align*}
which is again of the form \eqref{regular form}.  Recall from subsection \ref{key properties} that there are at most $2^{m_j}$ terms in the sum over $\beta_1$.  Repeating this argument $m_j-1$ more times yields the desired result \eqref{regular bound d ge 2}. 
\end{proof}
\end{prop}

Before we proceed with the proof of Lemma \ref{key estimate}, we present a short lemma that we use to bound the term $|\phi|^2\phi$ appearing on the right hand side of \eqref{distinguished bound d ge 2}.

\begin{lemma}\label{sobx2}
Let $\epsilon>0$.  Then, for $s_c=\frac{d}{2}-1$, $r_\epsilon=\frac{2d}{d+2(1-\epsilon)}$, and $d\ge 3$, we have
\begin{align}
\||\phi|^2\phi\|_{W^{-(s_c+\frac{\epsilon}{2}),r_\epsilon}}
\lesssim\|\phi\|^3_{H^{s_\epsilon}}.\label{cubic 3}
\end{align}
Similarly, when $d=2$, we have
\begin{align}
\||\phi|^2\phi\|_{W^{-(\frac{1}{3}-\frac{\epsilon}{2}),r_\epsilon}}
\lesssim\|\phi\|^3_{H^{1/3}}.\label{cubic 2}
\end{align}
\end{lemma}

\begin{proof}
Let $d\ge 3$.  By two applications of the Sobolev inequality, we have
\begin{equation*}
\||\phi|^2\phi\|_{W^{-(s_c+\frac{\epsilon}{2}),r_\epsilon}} \lesssim \||\phi|^2\phi\|_{L^{\frac{2d}{2d-\epsilon}}}=\|\phi\|_{L^{\frac{6d}{2d-\epsilon}}}^3 \lesssim \|\phi\|^3_{H^{\frac{d+\epsilon}{6}}}\le\|\phi\|^3_{H^{s_\epsilon}}.
\end{equation*}
This establishes \eqref{cubic 3}. The proof for the case $d=2$ is similar.

%Now let $d=2$.  In this case, the kernel $G_{\alpha}$ of the Bessel potential $<\nabla>^{-\alpha}$ behaves like $\frac{1}{|x|^{d-\alpha}}$ near $0$ and decays exponentially like $e^{-c|x|}$ at $\infty$ (see Chapter V \S3 of Stein's book \cite{SteSingular}) for some $c>0$. Thus
%\begin{align} 
%\||\phi|^2\phi\|_{W^{-(\frac{1}{3}-\frac{\epsilon}{2}),r_\epsilon}}&=\|<\nabla>^{-(\frac{1}{3}-\frac{\epsilon}{2})} (|\phi|^2\phi)\|_{L^{r_\epsilon}}=\|\int G_{\frac{1}{3}-\frac{\epsilon}{2}}(x-y) (|\phi|^2\phi)(y)dy\|_{L_x^{r_\epsilon}}\notag \\ 
%&\leq\int \|G_{\frac{1}{3}-\frac{\epsilon}{2}}(x-y)\|_{L_x^{r_\epsilon}} (|\phi|^3)(y)dy  \label{minkowski int} \\
%&\lesssim \int |\phi|^3dy \label{integrability of Bessel} \\
%&\lesssim \|\phi\|^3_{H^{1/3}}. \notag
%\end{align}
%To obtain \eqref{minkowski int} we use Minkowski integral inequality. \eqref{integrability of Bessel} follows from the integrability of the Bessel potential, since when $d=2$ and $\epsilon$ small, $(d-(\frac{1}{3}-\frac{\epsilon}{2}))r_{\epsilon}=(2-\frac{1}{3}+\frac{\epsilon}{2})\frac{2}{2-\epsilon}<2$. 
\end{proof}

We are now ready to conclude the proof of Theorem \ref{thm: main theorem} by proving Lemma \ref{key estimate}.

\begin{proof}[Proof of Lemma \ref{key estimate}]
Recall from \eqref{J decomposition} that $J^k$ can be decomposed into a product of $k$ one particle kernels
\begin{align*}
J^k(t,t_1,\dots,t_n;\sigma)=\prod_{j=1}^k J^1_j(t,t_{\ell_{j,1}},\dots,t_{\ell_{j,m_j}};\sigma_j),
\end{align*}
where only one of the factors $J_j^1$ distinguished.  It now follows from Propositions \ref{L^p integral} and \ref{H^s integral} that
\begin{align*}
&\int_{[0,T)^{n-1}}dt_1\cdots dt_{n-1}{\rm Tr}\bigg(\bigg|S^{(k,-d)}J^k(t,t_1,\dots,t_n;\sigma)\bigg|\bigg)\\
&=\int_{[0,T)^{n-1}}dt_1\cdots dt_{n-1}\prod_{j=1}^k{\rm Tr}\bigg(\bigg|S^{(1,-d)}J^1_j(t,t_{\ell_{j,1}},\dots,t_{\ell_{j,m_j}};\sigma_j)\bigg|\bigg)\\
&\le
\begin{cases}
2^nC^{n-1}T^{\epsilon(n-1)}\|\phi\|_{H^{s_\epsilon}}^{2(k+n)-3}\||\phi|^2\phi\|_{W^{-(s_c+\tfrac{\epsilon}{2}),r_\epsilon}}&\text{ if }d\ge 3\\
2^nC^{n-1}T^{\tfrac{1}{3}(n-1)}\|\phi\|_{H^{1/3}}^{2(k+n)-3}\||\phi|^2\phi\|_{W^{-(\tfrac{1}{3}-\tfrac{\epsilon}{2}),r_\epsilon}}&\text{ if }d=2\\
2^nC^{n-1}T^{\tfrac{1}{2}(n-1)}\|\phi\|_{L^2}^{2(k+n)-3}\||\phi|^2\phi\|_{L^1}&\text{ if }d=1.
\end{cases}
\end{align*}
Thus, for $t\in[0,T)$, it follows from Lemma \ref{sobx2} that
\begin{align*}
&\int_{[0,T)^{n-1}} d\underline{t}_{n-1} \textup{Tr}(|S^{(k,-d)}J^k(\underline{t}_n;\sigma)|)\\
&\le
\begin{cases}
(CT^\epsilon)^{n-1} \|\phi\|_{H^{s_\epsilon}}^{2(k+n)}&\text{ if }d\ge 3\\
(CT^{1/3})^{n-1} \|\phi\|_{H^{1/3}}^{2(k+n)}&\text{ if }d=2\\
(CT^{1/2})^{n-1}\|\phi\|_{H^{1/6}}^{2(k+n)}&\text{ if }d=1,
\end{cases}
\end{align*}
which is precisely the statement of Lemma \ref{key estimate}.
\end{proof}

\appendix

\section{Proof of Lemma \ref{lemma: trilinear}}  \label{sec: proof trilinear}

We prove Lemma \ref{lemma: trilinear} combining the dispersive estimate, the Strichartz estimates (see \cite{KeelTao} for example) and negative order Sobolev norms.

\begin{lemma}[Dispersive estimates] \label{lemma: dispersive estimates}
For $2\leq r\leq\infty$, we have
\begin{equation} \label{lp dispersive estimates}
 \|e^{it\Delta}f\|_{L_{x}^r} \lesssim |t|^{-d(\frac{1}{2}-\frac{1}{r})} \|f\|_{L_x^{r'}}.
\end{equation}
\end{lemma}

\begin{lemma}[Homogeneous Strichartz estimates] \label{lemma: strichartz estimates}
We call a pair of exponents $(q,r)$ Schr\"odinger admissible if $2\leq q,r\leq \infty$, $\frac{2}{q}+\frac{d}{r}=\frac{d}{2}$ and $(q,r,d)\neq (2,\infty,2)$. Then for any admissible exponents $(q,r)$ we have the homogeneous Strichartz estimate
\begin{equation}  \label{homo Strichartz estimate}
  \|e^{it\Delta}f\|_{L_t^q L_{x}^{r}} \lesssim \|f\|_{L_x^2}.
\end{equation}
\end{lemma}

\begin{lemma}[Negative order Sobolev norms]\label{lemma: negLeibniz} 
Let $\epsilon>0$ be a small number. Then, for $s\geq s_c+\frac{\epsilon}{2}$, we have
$$\|fg\|_{W^{-s,r_{\epsilon}}}\lesssim\|f\|_{W^{-s,r'_{\epsilon}}}\|g\|_{W^{s, \frac{2d}{d+2-3\epsilon}}},$$
where $r_{\epsilon}=\frac{2d}{d+2(1-\epsilon)}$.
\end{lemma}

\begin{proof}
By H\"older's inequality, the fractional Leibniz rule and the Sobolev inequality, we have
\begin{align*}
\Big|\int f(x)g(x)\overline{h(x)} dx\Big|&\leq \|f\|_{W^{-s, r'_{\epsilon}}}\|g\bar{h}\|_{W^{s, r_{\epsilon}}}\\
&\lesssim \|f\|_{W^{-s, r'_{\epsilon}}}\Big(\|g\|_{W^{s, \frac{2d}{d+2-3\epsilon}}}\|h\|_{L^{\frac{2d}{\epsilon}}}+\|g\|_{L^{\frac{d}{2(1-\epsilon)}}}\|h\|_{W^{s, r'_{\epsilon}}}\Big)\\
&\lesssim\|f\|_{W^{-s, r'_{\epsilon}}}\|g\|_{W^{s, \frac{2d}{d+2-3\epsilon}}}\|h\|_{W^{s, r'_{\epsilon}}}.
\end{align*}
The lemma now follows from the standard duality argument.
\end{proof}

\begin{proof}[Proof of Lemma \ref{lemma: trilinear}]
$(i)$. For notational convenience, we omit the time interval $[0,T)$ in the norms. \\
$(\ref{L^p bound})$: By Lemma \ref{lemma: negLeibniz}, we get
\begin{equation}   \label{dispersive flip}
\begin{aligned}
\|T(f,g,h)\|_{W^{-(s_c+\frac{\epsilon}{2}),r_{\epsilon}}}&\lesssim\|e^{i(t-t_1)\Delta}f\|_{W^{-(s_c+\frac{\epsilon}{2}),r'_{\epsilon}}}\|(e^{i(t-t_2)\Delta}g)(e^{i(t-t_3)\Delta}h)\|_{W^{s_c+\frac{\epsilon}{2}, \frac{2d}{d+2-3\epsilon}}}\\
&\lesssim\tfrac{1}{|t-t_1|^{1-\epsilon}}\|f\|_{W^{-(s_c+\frac{\epsilon}{2}),r_{\epsilon}}}\|g\|_{H^{s_{\epsilon}}}\|h\|_{H^{s_{\epsilon}}}.
\end{aligned}
\end{equation}
Here, in the second inequality, we use the dispersive estimate:
$$\|e^{i(t-t_1)\Delta}f\|_{W^{-(s_c+\frac{\epsilon}{2}),r'_{\epsilon}}}\lesssim\tfrac{1}{|t-t_1|^{1-\epsilon}}\|f\|_{W^{-(s_c+\frac{\epsilon}{2}),r_{\epsilon}}}$$
and the fractional Leibniz rule and the Sobolev inequality:
\begin{equation}
\begin{aligned}
&\|(e^{i(t-t_2)\Delta}g)(e^{i(t-t_3)\Delta}h)\|_{W^{s_c+\frac{\epsilon}{2}, \frac{2d}{d+2-3\epsilon}}}\\
&\lesssim\|e^{i(t-t_2)\Delta}g\|_{W^{s_c+\frac{\epsilon}{2}, \frac{2d}{d-\epsilon}}}\|e^{i(t-t_3)\Delta}h\|_{L^{ \frac{d}{1-\epsilon}}}+\|e^{i(t-t_2)\Delta}g\|_{L^{ \frac{d}{1-\epsilon}}}\|e^{i(t-t_3)\Delta}h\|_{W^{s_c+\frac{\epsilon}{2}, \frac{2d}{d-\epsilon}}}\\
&\lesssim\|e^{i(t-t_2)\Delta}g\|_{H^{s_{\epsilon}}}\|e^{i(t-t_3)\Delta}h\|_{H^{s_{\epsilon}}}=\|g\|_{H^{s_{\epsilon}}}\|h\|_{H^{s_{\epsilon}}}.
\end{aligned}
\end{equation}
Integrating out the time variable $t$, we prove \eqref{L^p bound}.\\ \\
$(\ref{H^s bound})$: By the fractional Leibniz rule, we have
\begin{align*}
\|T(f,g,h)\|_{L_t^1H_x^{s_{\epsilon}}}&\lesssim \|e^{i(t-t_1)\Delta}f\|_{L_t^3 W_x^{s_{\epsilon},\frac{6d}{3d-4}}}\|e^{i(t-t_2)\Delta}g\|_{L_t^3 L_x^{3d}}\|e^{i(t-t_3)\Delta}h \|_{L_t^3 L_x^{3d}}\\
&+\|e^{i(t-t_1)\Delta}f\|_{L_t^3 L_x^{3d}}\|e^{i(t-t_2)\Delta}g\|_{L_t^3 W_x^{s_{\epsilon},\frac{6d}{3d-4}}}\|e^{i(t-t_3)\Delta}h \|_{L_t^3 W_x^{3d}}\\
&+\|e^{i(t-t_1)\Delta}f\|_{L_t^3 L_x^{3d}}\|e^{i(t-t_2)\Delta}g\|_{L_t^3 L_x^{3d}}\|e^{i(t-t_3)\Delta}h \|_{L_t^3 W_x^{s_{\epsilon},\frac{6d}{3d-4}}}.
\end{align*}
Then, by the Sobolev inequality and the Strichartz estimates, we bound the first term by
\begin{align*}
&\lesssim\|e^{i(t-t_1)\Delta}f\|_{L_t^3 W_x^{s_{\epsilon},\frac{6d}{3d-4}}}\|e^{i(t-t_2)\Delta}g\|_{L_t^3 W_x^{s_{\epsilon},\frac{6d}{3d-4+6\epsilon}}}\|e^{i(t-t_3)\Delta}h \|_{L_t^3 W_x^{s_{\epsilon},\frac{6d}{3d-4+6\epsilon}}}\\
&\leq T^\epsilon\|e^{i(t-t_1)\Delta}f\|_{L_t^3 W_x^{s_{\epsilon},\frac{6d}{3d-4}}}\|e^{i(t-t_2)\Delta}g\|_{L_t^{\frac{6}{2-3\epsilon}} W_x^{s_{\epsilon},\frac{6d}{3d-4+6\epsilon}}}\|e^{i(t-t_3)\Delta}h \|_{L_t^{\frac{6}{2-3\epsilon}} W_x^{s_{\epsilon},\frac{6d}{3d-4+6\epsilon}}}\\
&\lesssim T^\epsilon \|f\|_{H^{s_{\epsilon}}}\|g\|_{H^{s_{\epsilon}}}\|h\|_{H^{s_{\epsilon}}}.
\end{align*}
Similarly, we bound the other two terms. \\

$(ii)$. 
\eqref{2d bound}: The proof is similar to that of $(\ref{L^p bound})$, but here we use Lemma \ref{lemma: negLeibniz} with $s=(\frac{1}{3}-\frac{\epsilon}{2})$. Indeed, by the dispersive estimate and Lemma \ref{lemma: negLeibniz},
\begin{align*}
\|T(f,g,h)\|_{W^{-(\frac{1}{3}-\frac{\epsilon}{2}),r_{\epsilon}}}&\lesssim\|e^{i(t-t_1)\Delta}f\|_{W^{-(\frac{1}{3}-\frac{\epsilon}{2}),r'_{\epsilon}}}\|(e^{i(t-t_2)\Delta}g)(e^{i(t-t_3)\Delta}h)\|_{W^{\frac{1}{3}-\frac{\epsilon}{2}, \frac{2d}{d+2-3\epsilon}}}\\
&\lesssim\tfrac{1}{|t-t_1|^{1-\epsilon}}\|(e^{i(t-t_2)\Delta}g)(e^{i(t-t_3)\Delta}h)\|_{W^{\frac{1}{3}-\frac{\epsilon}{2}, \frac{2d}{d+2-3\epsilon}}}.
\end{align*}
Then, modifying $(A.1)$, we obtain
\begin{align*}
&\|(e^{i(t-t_2)\Delta}g)(e^{i(t-t_3)\Delta}h)\|_{W^{\frac{1}{3}-\frac{\epsilon}{2}, \frac{2d}{d+2-3\epsilon}}}\\
&\lesssim\|e^{i(t-t_2)\Delta}g\|_{W^{\frac{1}{3}-\frac{\epsilon}{2} \frac{2d}{d-\epsilon}}}\|e^{i(t-t_3)\Delta}h\|_{L^{ \frac{d}{1-\epsilon}}}+\|e^{i(t-t_2)\Delta}g\|_{L^{ \frac{d}{1-\epsilon}}}\|e^{i(t-t_3)\Delta}h\|_{W^{\frac{1}{3}-\frac{\epsilon}{2}, \frac{2d}{d-\epsilon}}}\\
&\lesssim\|e^{i(t-t_2)\Delta}g\|_{H^{1/3}}\|e^{i(t-t_3)\Delta}h\|_{H^{1/3}}=\|g\|_{H^{1/3}}\|h\|_{H^{1/3}},
\end{align*}
Applying this to the above inequality and Integrating out $t$, we compete the proof. \\

\eqref{2d H13 bound}: Although we set $\epsilon$ to be small and $d\geq 3$ in the proof of \eqref{H^s bound}, it actually works for $\epsilon=\frac{1}{3}$ and $d=2$ which is exactly \eqref{2d H13 bound}. \\

$(iii)$.
For \eqref{L^1 bound}, by the H\"older inequality and the 1d dispersive estimates, we get
$$\|T(f,g,h)\|_{L^1}\leq \|e^{i(t-t_1)}f\|_{L^{\infty}} \|e^{i(t-t_2)}g\|_{L^2} \|e^{i(t-t_3)}h\|_{L^2} \lesssim \tfrac{1}{|t-t_1|^{1/2}} \|f\|_{L^1} \|g\|_{L^2} \|h\|_{L^2}.$$
Integrating out the time variable $t$, we prove \eqref{L^1 bound}.\\
For \eqref{L^2 bound}, by the H\"older inequality and the Strichartz estimate,
$$\|T(f,g,h)\|_{L_t^1 L_x^2}\leq T^{1/2}\|e^{i(t-t_1)}f\|_{L_{t,x}^6} \|e^{i(t-t_2)}g\|_{L_{t,x}^6} \|e^{i(t-t_3)}h\|_{L_{t,x}^6}\lesssim T^{1/2}\|f\|_{L^2}\|g\|_{L^2}\|h\|_{L^2}.$$
\end{proof}

\begin{acklg}
 The authors would like to express their special appreciation and thanks to their mentors Thomas Chen and Nata\v{s}a Pavlovi\'c for proposing the problem and for various useful discussions.
\end{acklg}

%%%%%%%%%%%%%%%%%%%%%%%%%%%%%%%%%%
\bibliographystyle{abbrv}

%%%%%%%%%%%%%%%%%%%%%%%%%%%%%%%%%%
\end{document}